\documentclass[reqno,11pt]{amsart}
\usepackage[colorlinks=true, linkcolor=blue, citecolor=blue]{hyperref}

\usepackage{amssymb}
\usepackage{amsmath, graphicx, rotating}
\usepackage{color}
\usepackage{soul}
\usepackage[dvipsnames]{xcolor}

\usepackage{ifthen}
\usepackage{xkeyval}
\usepackage{todonotes}
\usepackage[T1]{fontenc}
\usepackage{lmodern}
\usepackage[english]{babel}

\usepackage{ upgreek }
\usepackage{stmaryrd}
\SetSymbolFont{stmry}{bold}{U}{stmry}{m}{n}
\usepackage{amsthm}
\usepackage{float}

\usepackage{ bbm }
\usepackage{ stmaryrd }
\usepackage{ mathrsfs }
\usepackage{ frcursive }
\usepackage{ comment }

\usepackage{pgf, tikz}
\usetikzlibrary{shapes}
\usepackage{varioref}
\usepackage{enumitem}

\usepackage{mathtools}

\definecolor{rouge}{rgb}{0.7,0.00,0.00}
\definecolor{vert}{rgb}{0.00,0.5,0.00}
\definecolor{bleu}{rgb}{0.00,0.00,0.8}

\usepackage[textheight=7.61in,textwidth=4.98in]{geometry} 
\newtheorem{theorem}{Theorem}[section]
\newtheorem*{theorem*}{Theorem}
\newtheorem{lemma}[theorem]{Lemma}

\newtheorem{corollary}[theorem]{Corollary}
\newtheorem{proposition}[theorem]{Proposition}

\labelformat{hypothesis}{\textbf{M\kern-0.1mm#1}}

\newtheorem{condition}{Condition}
\newtheorem{conditionA}{A\kern-0.1mm}
\labelformat{conditionA}{\textbf{A\kern-0.1mm#1}}

\renewcommand\dots{\hbox to 1em{.\hss.\hss.}}

\theoremstyle{definition}

\numberwithin{equation}{section}

\def\bb#1{\mathbb{#1}}

\def\bbm#1{\mathbbm{#1}}

\def\geq{\geqslant}
\def\leq{\leqslant}

\newcommand\ee{\varepsilon}

\DeclareMathOperator{\supp}{supp}

\DeclarePairedDelimiter\floor{\lfloor}{\rfloor}

\begin{document}

\title[Moderate deviations for products of random matrices]
{Berry-Esseen bounds and moderate deviations 
 for the norm, entries and spectral radius of  products of positive random matrices}

\author{Hui Xiao}
\author{Ion Grama}
\author{Quansheng Liu}

\curraddr[Xiao, H.]{ Universit\'{e} de Bretagne-Sud, LMBA UMR CNRS 6205, Vannes, France}
\email{hui.xiao@univ-ubs.fr}
\curraddr[Grama, I.]{ Universit\'{e} de Bretagne-Sud, LMBA UMR CNRS 6205, Vannes, France}
\email{ion.grama@univ-ubs.fr}
\curraddr[Liu, Q.]{ Universit\'{e} de Bretagne-Sud, LMBA UMR CNRS 6205, Vannes, France}
\email{quansheng.liu@univ-ubs.fr}


\begin{abstract}
Let $(g_{n})_{n\geq 1}$ be a sequence of independent and identically distributed 
positive random $d\times d$ matrices and consider the matrix product  $G_n = g_n \ldots g_1$.
Under suitable conditions, we establish  
the Berry-Esseen bounds on the rate of convergence in the central limit theorem and  
moderate deviation expansions of Cram\'er type,  
for the matrix norm $\| G_n \|$ of $G_n$, for its $(i,j)$-th entry $G_n^{i,j}$ 
and for its spectral radius $\rho(G_n)$. 
\end{abstract}

\date{\today}
\subjclass[2010]{Primary 60F05, 60F10, 60B20; Secondary 60J05, 60B15.}
\keywords{Berry-Esseen bound; Cram\'{e}r type moderate deviation;
Products of random matrices; Operator norm; Entries; Spectral Radius}

\maketitle

\section{Introduction}
Fix an integer $d \geq 2$. 
Let $(g_{n})_{n\geq 1}$ be a sequence of independent and identically distributed (i.i.d.)  
positive random $d\times d$ matrices of the same probability law $\mu$.
Set $G_n = g_n \ldots g_1$ and denote by $\| G_n \|$ any matrix norm of the product $G_n$.   
It has been of great interest in recent years to investigate the asymptotic behaviors 
of the random matrix product $G_n$ since the pioneering work of Furstenberg and Kesten \cite{FK60}. 
In \cite{FK60}  the strong law of large numbers (SLLN) for the matrix norm $\| G_n \|$ was established:  
if $\mathbb{E} ( \max \{ 0, \log \| g_1 \| \} ) < \infty$, then 
\begin{align} \label{Intro_LLN}
\lim_{ n \to \infty } \frac{1}{n} \log \| G_n \| = \lambda, \quad  a.s.,
\end{align}
where $\lambda$ is  a constant called the upper Lyapunov exponent of the product $G_n$.
This result can be seen as a direct consequence of Kingman's subaddtive ergodic theorem \cite{Kin73}.  
The central limit theorem (CLT) for $\| G_n \|$ was also proved in \cite{FK60}: 
for any $y \in \mathbb{R}$, 
\begin{align}\label{Intro-CLT}
\lim_{n \to \infty} \bb{P} \left(
 \frac{ \log \| G_n \| - n \lambda }{ \sigma \sqrt{n} } \leq y  \right)
= \frac{1}{\sqrt{2\pi}} \int_{-\infty}^y e^{-\frac{t^2}{2}} dt =: \Phi(y),
\end{align}
where $\sigma^2 > 0$ is the asymptotic variance corresponding to the product $G_n$. 
The conditions used in  \cite{FK60} for the proof of  \eqref{Intro-CLT} 
have been relaxed later by Hennion \cite{Hen97}  to the second moment condition 
together with the allowability and positivity condition that we will present in the next section. 
We mention that in the case of invertible random matrices, the CLT \eqref{Intro-CLT}
was established by Le Page \cite{LeP82}, and has been extended by Goldsheid and Guivarc'h \cite{GG96} 
to a multidimensional version, and by Benoist and Quint \cite{BQ16a} 
to the general framework of reductive groups under optimal moment conditions.

In \cite{XGL19b} the authors proved a Berry-Esseen bound and a moderate deviation expansion
for the norm cocycle $\log | G_n x |$ jointly with the Markov chain $X_n^x= G_n x /| G_n x |$, 
where $x$ is any starting point on the unit sphere and $| \cdot |$ 
is the euclidean norm in $\bb R^d$. 
For related results  about the vector norm  $ | G_n x |$  we refer to  
\cite{LeP82, BL85, Aou11, Gui15, BM16, BQ16b, CDJ17, CDM17, Ser19, XGL19a}. 
However, this type of results 
for other important quantities like the matrix norm  $\|G_n\|$, the entries $G_n^{i,j}$  and the 
spectral radius $\rho(G_n)$ of $G_n$ are  
absent in the literature.
The goal of the present  paper is to fill this gap by extending the results of \cite{XGL19b}  
to the matrix norm, to the entries and to the 
spectral radius  for the product $G_n$ of positive random matrices, 
jointly with the Markov chain $(X_n^x)_{n \geq 0}$.

Let us explain briefly the main results that we obtain for the matrix norm.  
We would like to quantify the error in the normal approximation \eqref{Intro-CLT}.
We do this in two ways.  
The first way is to estimate the absolute error. 
In this spirit, under suitable conditions we prove the following Berry-Esseen bound: 
there exists a constant $C>0$ such that for all $n \geq 1$, 
\begin{align}\label{Intro-BEThm}
\sup_{y \in \mathbb{R}} 
\left| \bb{P} \left(
 \frac{\log  \| G_n \| - n \lambda }{ \sigma \sqrt{n}} \leq y \right)
 - \Phi(y) \right| \leq \frac{C}{\sqrt{n}}.
\end{align}
Our result \eqref{Intro-BEThm} is clearly a refinement of \eqref{Intro-CLT} by giving the rate of convergence.
In fact, a more general version of the Berry-Esseen bound for the couple $(X_n^x, \log \|G_n\|)$
with a target function $\varphi$ on $X_n^x$ is given in Theorem \ref{Thm-BerryE-Norm}.  


The second way is to study the relative error in \eqref{Intro-CLT}. Along this line we prove the following
Cram\'{e}r type moderate deviation expansion: as $n \to \infty$, 
 uniformly in $y \in [0, o(\sqrt{n} )]$, 
\begin{align}\label{Intro-Cram-Posi}
\frac{\mathbb{P}
  \left( \frac{\log \| G_n \| - n \lambda }{ \sigma \sqrt{n} } \geq y  \right)
    } {1-\Phi(y)}
= e^{ \frac{y^3}{\sqrt{n}}\zeta ( \frac{y}{\sqrt{n}} ) }
  \left[ 1 + O \left( \frac{ y+1 }{ \sqrt{n} } \right) \right],
\end{align}
where $\zeta$ is the Cram\'{e}r series (see \eqref{Def-CramSeri}).
Note that the expansion \eqref{Intro-Cram-Posi} clearly implies the moderate deviation principle 
for the matrix norm $\| G_n \|$, see Corollary \ref{Ch5_coro-mdp-norm}, 
which to the best of our knowledge was not known before.



The results \eqref{Intro-BEThm} and \eqref{Intro-Cram-Posi} concern the matrix norm $\|G_n\|$,
but we also prove that they remain valid (under stronger conditions) when the matrix norm $\|G_n\|$
is replaced by the entries $G_n^{i,j}$ or the spectral radius $\rho(G_n)$: 
see Theorems \ref{Thm-BerryE-Posi} and \ref{Thm-Cram-Posi-tag}. 
The corresponding strong law of large numbers and the central limit theorem were established in \cite{FK60, CNP93, Hen97}
for the entries $G_n^{i,j}$, and in \cite{Hen97} for the spectral radius $\rho(G_n)$. 
However,  our Theorems \ref{Thm-BerryE-Posi} and \ref{Thm-Cram-Posi-tag} on Berry-Esseen bounds and   
Cram\'{e}r type moderate deviation expansions
for the entries $G_n^{i,j}$ and the spectral radius $\rho(G_n)$ are new.

The proofs of  \eqref{Intro-BEThm} and \eqref{Intro-Cram-Posi}
are based on the recent results established  in \cite{XGL19b}  about  
the Berry-Esseen bound and the Cram\'{e}r type moderate deviation expansion for the norm cocyle $\log |G_n x|$
and on a comparison between $\| G_n \|$ and $|G_n x|$ (Lemma \ref{Lem_Norm}),
where $x$ is a vector in $\mathbb{R}^d$ with strictly positive components. 

To prove \eqref{Intro-BEThm} and \eqref{Intro-Cram-Posi}
when the matrix norm $\| G_n \|$ is replaced by the entries $G_n^{i,j}$, 
in addition to the use  of the aforementioned results established  in \cite{XGL19b}, we do 
 a careful quantitative analysis of the comparison between 
$\log G_n^{i,j}:= \log \langle e_i, G_n e_j \rangle$ and $\log |G_n e_j|$,
where $(e_i)_{ 1 \leq k \leq d}$ is the canonical orthonormal basis in $\mathbb{R}^d$. 
This comparison is possible due to a regularity condition which ensures that all the entries 
in the same column of the matrix $g \in \supp \mu$ (the support of $\mu$) are comparable: 
see condition \ref{Condi-KestenH}.
Note that this condition is  
weaker than the Furstenberg-Kesten condition \eqref{Intro-Condi-Kesten} used in \cite{FK60},
which says that all the entries of the matrix $g \in \supp \mu$ are comparable. 



Using the results mentioned above for the matrix norm $\| G_n \|$ and for the vector norm 
$|G_n x|$ established in \cite{XGL19b}, 
we then prove the corresponding results for the spectral radius $\rho(G_n)$
based on the Collatz-Wielandt formula: see Theorems \ref{Thm-BerryE-Posi} and \ref{Thm-Cram-Posi-tag}. 

When the boundedness condition \ref{Condi-KestenH} of Furstenberg-Kesten  type 
is relaxed to a moment condition  \ref{Condi-WeakKest}, 
we are also able to establish Berry-Esseen type bounds and moderate deviation principles 
for the entries $G_n^{i,j}$ and the spectral radius $\rho(G_n)$: 
see Theorems \ref{Thm-BerryE-Posi_Weak} and \ref{ciro-LDP002}. 
Note that under condition \ref{Condi-WeakKest}, 
the Markov chain $(X_n^x)_{n \geq 0}$ is no longer separated from the coordinates $e_i$
and an important step to prove Theorems \ref{Thm-BerryE-Posi_Weak} and \ref{ciro-LDP002}
is to establish the H\"{o}lder regularity of the stationary measure $\nu$ shown in Proposition \ref{PropRegularity},
which is also of independent interest. 
The proof of Proposition \ref{PropRegularity} is based on the large deviation bounds 
for the norm cocycle $\log |G_n x|$ stated in Theorem \ref{Thm_BRLD_changedMea}.  


In closing this section, we mention that Berry-Esseen bounds and moderate  deviations 
for random matrices on different aspects  have been considered in the literature,  
see e.g. Chen, Gao and Wang \cite{CGW15}  for eigenvalues of a single random matrix when the dimension goes to $\infty$.

\section{Main results}

\subsection{Notation and conditions}

For any integer $d\geq 2$,
denote by $\mathcal M_+$ the multiplicative semigroup of $d\times d$ matrices 
with non-negative entries in $\mathbb{R}$.
A non-negative matrix $g \in \mathcal M_+$ is said to be \emph{allowable},
if every row and every column of $g$ contains a strictly positive entry.
We write $\mathcal M_+^\circ $ for the subsemigroup of $\mathcal M_+$ with strictly positive entries.
Equip the space $\mathbb{R}^{d}$ with the standard scalar product $\langle \cdot, \cdot \rangle$
and the Euclidean norm $|\cdot|$. For a vector $x$, we write $x\geq 0$ (resp. $x>0$) if all its components are non-negative (resp. strictly positive). 
Denote by $\mathbb{S}^{d-1}_{+} = \{x \geq 0 : |x|=1 \}$ the intersection of the unit sphere
with the positive quadrant.
The space $\mathbb{S}^{d-1}_{+}$ is endowed with the Hilbert cross-ratio metric $\mathbf{d}$, 
i.e., for any $x=(x_1, \ldots, x_d)$ and $y=(y_1, \ldots, y_d)$ in  $\mathbb{S}_{+}^{d-1}$,
\begin{align*}
\mathbf{d}(x,y) = \frac{ 1- m(x,y)m(y,x) }{ 1 + m(x,y)m(y,x) }, 
\end{align*} 
where 
\begin{align*}
m(x,y) = \sup \left\{ \alpha > 0:  \alpha y_i\leq x_i,\  \forall i=1,\ldots, d  \right\}. 
\end{align*}
It is shown in \cite{Hen97} that
there exists a constant $C>0$ such that 
$|x-y|\leq C \mathbf{d}(x,y)$ for any $x,y \in \mathbb{S}_{+}^{d-1}$. 
We refer to  \cite{Hen97} for more properties of the metric $\mathbf{d}$.

Let $\mathcal{C}(\mathbb{S}^{d-1}_{+})$ be the space of continuous complex-valued functions 
on $\mathbb{S}^{d-1}_{+}$ and $\mathbf{1}$ be the constant function  
with value $1$.
Throughout the paper we always assume that $\gamma>0$ is a fixed small enough constant.  
For any $\varphi\in \mathcal{C}(\mathbb{S}^{d-1}_{+})$, set
\begin{align*}
&\|\varphi\|_{\gamma}:= \|\varphi\|_{\infty} + [\varphi]_{\gamma}, \ \   
\|\varphi\|_{\infty}:=  \!\!  \sup_{x \in \mathbb{S}^{d-1}_{+} }|\varphi(x)|,  \ \
[\varphi]_{\gamma}: =  \!\!
  \sup_{x,y\in \mathbb{S}^{d-1}_{+} }\frac{|\varphi(x)-\varphi(y)|}{\mathbf{d}^{\gamma }(x,y)}. 
\end{align*}
We introduce the Banach space
\begin{align*}
\mathcal{B}_{\gamma}:= \left\{ \varphi\in \mathcal{C}(\mathbb{S}^{d-1}_{+}): \| \varphi \|_{\gamma}< + \infty \right\}.
\end{align*}

Let $(g_{n})_{n\geq 1}$ be a sequence of i.i.d. positive random matrices 
of the same probability law $\mu$ on $\mathcal M_+$.
Denote by $\supp \mu$ the support of the measure $\mu$. 
Consider the matrix product $G_n = g_n \ldots g_1$ 
and denote by $G_n^{i,j}$ the $(i,j)$-th entry of $G_n$, where $1 \leq i, j \leq d$. 
It holds that 
\begin{align*}
G_n^{i,j} = \langle e_i, G_n e_j \rangle, 
\end{align*}
where $(e_k)_{1 \leq k \leq d}$ is the canonical orthonormal basis of $\mathbb{R}^d$. 
For any $g \in \mathcal M_+$, denote by $\rho(g)$ the spectral radius of $g$,
and by $\| g \|$ its operator norm, i.e., $\| g \| = \sup_{ x\in \mathbb{S}_+^{d-1} }|g x|$.
By Gelfand's formula, it holds that $\rho(g) = \lim_{k \to \infty} \| g^k \|^{1/k}$. 
In this paper, we are interested in Berry-Esseen bounds and moderate deviation
asymptotics for the matrix norm $\| G_n \|$, the entries $G_n^{i,j}$ and the spectral radius $\rho(G_n)$. 

Let $\iota(g) = \inf_{x\in \mathbb{S}_+^{d-1} }|gx|$ and $N(g) = \max \{ \|g\|, \iota(g)^{-1} \}$. 
We shall need the following exponential moment condition:

\begin{conditionA}\label{Condi-MomentH}
There exists a constant $\eta \in (0,1)$ such that $\mathbb{E} [ N(g_1)^{\eta} ] < +\infty$.
\end{conditionA}


Let $\Gamma_{\mu}$ be the smallest closed subsemigroup of $\mathcal M_+$ generated by $\supp \mu$. 
We will use the allowability and positivity conditions:
\begin{conditionA}\label{CondiAPH}
{\rm (i) (Allowability) }
Every $g\in\Gamma_{\mu}$ is allowable.

{\rm (ii) (Positivity) }
$\Gamma_{\mu}$ contains at least one matrix belonging to $\mathcal M_+^\circ$.
\end{conditionA}

It follows from the Perron-Frobenius theorem that every  $g \in \mathcal M_+^\circ$
has a dominant eigenvalue which coincides with its spectral radius $\rho(g)$. 
The corresponding eigenvector is denoted by $v_g$. 
It is easy to see that $v_g \in \mathbb{S}_+^{d-1}$.


The following condition ensures that all the entries 
in each column of the matrix $g \in \supp \mu$ are comparable. 
\begin{conditionA}\label{Condi-KestenH} 
For any $1\leq j \leq d$, 
there exists a constant $C > 1$ such that for any $g = (g^{i,j})_{1\leq i, j \leq d} \in \supp \mu$, 
\begin{align*}  
1 \leq \frac{\max_{1\leq i \leq d} g^{i,j} }{ \min_{1\leq i \leq d} g^{i,j} }  \leq C. 
\end{align*} 
\end{conditionA}

Note that the set of such type of matrices forms a subsemigroup of $\mathcal M_+$, 
because if two positive matrices $g_1$ and $g_2$ satisfy condition \ref{Condi-KestenH},
then so does the product $g_2 g_1$, as will be seen from Lemma \ref{lem equiv Kesten}
where an equivalent description of condition \ref{Condi-KestenH} will be provided.

It is easy to see that condition \ref{Condi-KestenH} implies condition \ref{CondiAPH}.  
However, our condition \ref{Condi-KestenH} is clearly weaker than the 
Furstenberg-Kesten condition used in \cite{FK60}: there exists a constant $C > 1$ such that
for any $g = (g^{i,j})_{1\leq i, j \leq d} \in \supp \mu$,
\begin{align}\label{Intro-Condi-Kesten}
1 \leq \frac{\max_{1\leq i, j\leq d}  g^{i,j} }{ \min_{1\leq i,j\leq d}  g^{i,j} } \leq C.   
\end{align}
This condition plays an essential role in \cite{FK60} for the proofs of 
the strong law of large numbers and the central limit theorem for entries $G_n^{i,j}$.

The following condition concerns the existence 
of the harmonic moments of the entries  of $g_1$: 
\begin{conditionA}\label{Condi-WeakKest}
For any $1 \leq i, j \leq d$, there exists a constant $\delta > 0$ such that 
\begin{align*}
\mathbb{E} \Big[ \big( g_1^{i,j} \big)^{-\delta} \Big]  < \infty.
\end{align*}
\end{conditionA}

Condition \ref{Condi-WeakKest} is used to establish Berry-Esseen type bounds
and moderate deviation principles for the entries $G_n^{i,j}$ and the spectral radius $\rho(G_n)$,
see Theorems \ref{Thm-BerryE-Posi_Weak} and \ref{ciro-LDP002},
where condition \ref{Condi-KestenH} is not assumed. 
Note that the conditions \ref{Condi-KestenH} and \ref{Condi-WeakKest} do not imply each other.
However, under the moment assumption \ref{Condi-MomentH},
condition \ref{Condi-KestenH} (and therefore also \eqref{Intro-Condi-Kesten}) implies condition \ref{Condi-WeakKest}.
 The converse is not true.

For any $x \in \mathbb{S}_+^{d-1}$ and allowable matrix $g \in \mathcal M_+$,
we write $g \cdot x: = \frac{gx}{ |gx| }$ for the projective action of the matrix $g$ 
on the projective space $\mathbb{S}_+^{d-1}$.
 For any starting point $x \in \mathbb{S}_+^{d-1}$, set $X_0^x = x$ and 
\begin{align*}
X_n^x = G_n \cdot x, \quad n \geq 1. 
\end{align*}
Then $(X_n^x)_{n \geq 0}$ forms a Markov chain on $\mathbb{S}_+^{d-1}$ with the transfer operator $P$ given as follows: 
for any $\varphi \in \mathcal{C}(\mathbb{S}^{d-1}_{+})$, 
\begin{align}\label{Ch5_Def_P_qwe}
P \varphi (x) = \int_{ \Gamma_{\mu} } \varphi(g \cdot x) \mu(dg), \quad  x \in \mathbb{S}^{d-1}_{+}. 
\end{align}
Under conditions \ref{Condi-MomentH} and \ref{CondiAPH}, 
the Markov chain $(X_n^x)_{n \geq 0}$ possesses a unique stationary probability measure 
$\nu$ on $\mathbb{S}_+^{d-1}$ such that for any $\varphi \in \mathcal{C}(\mathbb{S}^{d-1}_{+})$,
\begin{align*} 
\int_{ \mathbb{S}_+^{d-1} } \int_{ \Gamma_{\mu} } \varphi(g \cdot x) \mu(dg) \nu(dx) 
= \int_{ \mathbb{S}_+^{d-1} } \varphi(x) \nu(dx).
 \end{align*}
Moreover, the support of $\nu$ is given by
$\supp \nu = \overline{\{v_{g}\in \mathbb S^{d-1}_+: g\in \Gamma_{\mu} \cap \mathcal M_+^\circ \}}$. 
We refer to \cite{Kes73, Hen97, BDGM14, XGL19b} for more details. 

Under conditions \ref{Condi-MomentH} and \ref{CondiAPH}, 
it is shown in \cite{XGL19b} that uniformly in $x \in \mathbb{S}_+^{d-1}$, 
\begin{align}\label{Formu_sig}
\sigma^2 : = \lim_{n\to\infty} \frac{1}{n} \mathbb{E} \left[ (\log |G_n x| - n \lambda )^{2}  \right] 
\in [0, \infty),
\end{align}
where $\lambda$ is the upper Lyapunov exponent defined by \eqref{Intro_LLN}. 
Equivalent formulations of $\sigma^2$ will be given in Proposition \ref{Prop_Variance}. 
We shall need the following conditions. 
\begin{conditionA}\label{Condi-VarianceH}
The asymptotic variance $\sigma^2$ satisfies $\sigma^2 > 0$.
\end{conditionA}

%
%

\begin{conditionA}\label{Condi-NonArith}
{\rm (Non-arithmeticity)}
For $t>0$, $\theta \in [0, 2\pi)$ and a function $\varphi: \bb S_+^{d-1} \to \mathbb{R}$, the equation 
\begin{align*}
|gx|^{it} \varphi (g \cdot x) = e^{i\theta} \varphi(x),  \quad \forall g \in \Gamma_{\mu},  \forall  x\in \supp \nu,
\end{align*}
has no trivial solution except that $t = 0$, $\theta =0$ and $\varphi$ is a constant.
\end{conditionA}

Note that condition \ref{Condi-NonArith} implies \ref{Condi-VarianceH}. 
If the additive subgroup of $\mathbb{R}$ generated by the set
$\{ \log \rho(g) : g\in \Gamma_{\mu} \cap \mathcal M_+^\circ \}$
is dense in $\mathbb{R}$, then both conditions \ref{Condi-VarianceH} and \ref{Condi-NonArith} are fulfilled (see \cite{BM16}).
This sufficient condition was introduced by Kesten \cite{Kes73} and is usually easier to verify in practice.

\subsection{Berry-Esseen bounds}
The goal of this section is to present our results on the Berry-Esseen bounds for the matrix norm $\| G_n \|$, 
the entries $G_n^{i,j}$ and the spectral radius $\rho(G_n)$. 
Let us first state the result for the operator norm $\| G_n \|$. 
Denote $(\mathbb{S}_+^{d-1})^{\circ} = \{x > 0 : |x|=1 \}$, 
which is the interior of the projective space $\mathbb{S}_+^{d-1}$. 

\begin{theorem}\label{Thm-BerryE-Norm}
Assume conditions \ref{Condi-MomentH}, \ref{CondiAPH} and \ref{Condi-VarianceH}.  
Then, for any compact set $K \subset (\mathbb{S}_+^{d-1})^{\circ}$, 
there exists a constant $C>0$ such that for all $n \geq 1$ and $\varphi \in \mathcal{B}_{\gamma}$, 
\begin{align} \label{BerryE-Posi02}
\sup_{y \in \mathbb{R}} \sup_{ x \in K } 
\left| \mathbb{E} \left[ \varphi(X_n^x)
\mathbbm{1}_{ \big\{ \frac{\log \| G_n \| - n \lambda }{\sigma \sqrt{n}} \leq y \big\} }
   \right]
 - \nu(\varphi) \Phi(y) \right| \leq  \frac{C}{\sqrt{n}}  \| \varphi \|_{\gamma}.
\end{align}
\end{theorem}
Since all matrix norms are equivalent, it can be easily checked that in Theorem \ref{Thm-BerryE-Norm}, the operator norm $\|\cdot\|$ can be replaced by any matrix norm.

It would be interesting to show that \eqref{BerryE-Posi02}
holds uniformly in $x \in \mathbb{S}_+^{d-1}$ instead of $x \in K$. 
Note that Theorem \ref{Thm-BerryE-Norm} is proved under the exponential moment condition \ref{Condi-MomentH}.
It is not clear how to establish Theorem \ref{Thm-BerryE-Norm} under 
the polynomial moment condition on the matrix law $\mu$. 

If the stronger condition \ref{Condi-KestenH} holds instead of condition \ref{CondiAPH},
then we are able to prove the following Berry-Esseen bounds for the scalar product $\langle f, G_n x \rangle$
and for the spectral radius $\rho(G_n)$.

\begin{theorem}\label{Thm-BerryE-Posi}
Assume conditions \ref{Condi-MomentH}, \ref{Condi-KestenH} and \ref{Condi-VarianceH}. 
\begin{itemize}
\item[\rm{(1)}] There exists a constant $C>0$ such that for all $n \geq 1$ and $\varphi \in \mathcal{B}_{\gamma}$, 
\begin{align} 
\qquad \quad  \sup_{y \in \mathbb{R}} \sup_{f, x \in \mathbb{S}_+^{d-1} }
\left| \mathbb{E} \Big[ \varphi(X_n^x)
\mathbbm{1}_{ \big\{ \frac{\log \langle f, G_n x \rangle - n \lambda }{\sigma \sqrt{n}} \leq y \big\} }
   \Big]
 - \nu(\varphi) \Phi(y) \right| \leq   \frac{C}{\sqrt{n}} \| \varphi \|_{\gamma}.    \label{BerryE01}  
\end{align}
\item[\rm{(2)}] For any compact set $K \subset (\mathbb{S}_+^{d-1})^{\circ}$, 
there exists a constant $C>0$ such that for all $n \geq 1$ and $\varphi \in \mathcal{B}_{\gamma}$, 
\begin{align}
\qquad \sup_{y \in \mathbb{R}} \sup_{ x \in K }
 \left| \mathbb{E} \Big[ \varphi(X_n^x)
  \mathbbm{1}_{ \big\{ \frac{\log \rho(G_n) - n \lambda }{\sigma \sqrt{n}} \leq y \big\} }
   \Big]
   - \nu(\varphi) \Phi(y) \right| \leq  \frac{C}{\sqrt{n}}  \| \varphi \|_{\gamma}.  \label{BerryE02}
\end{align}
\end{itemize} 
\end{theorem}

In particular, taking $\varphi = \mathbf{1}$, $f = e_i$ and $x = e_j$ in \eqref{BerryE01},
we get the Berry-Esseen bound for the entries $G_n^{i,j}$.
The Berry-Esseen bounds \eqref{BerryE01} and \eqref{BerryE02} are new. 

If condition \ref{Condi-KestenH} is replaced by the weaker one \ref{Condi-WeakKest}
(under \ref{Condi-MomentH}, condition \ref{Condi-WeakKest} is weaker than \ref{Condi-KestenH}),
then we are able to establish the following result. 

\begin{theorem}\label{Thm-BerryE-Posi_Weak}
Assume conditions \ref{Condi-MomentH}, \ref{Condi-WeakKest} and \ref{Condi-NonArith}. 
\begin{itemize}
\item[\rm{(1)}] There exists a constant $C>0$ such that for all $n \geq 1$ and $\varphi \in \mathcal{B}_{\gamma}$, 
\begin{align} 
\qquad \quad  \sup_{y \in \mathbb{R}} \sup_{f, x \in \mathbb{S}_+^{d-1} }
\left| \mathbb{E} \Big[ \varphi(X_n^x)
\mathbbm{1}_{ \big\{ \frac{\log \langle f, G_n x \rangle - n \lambda }{\sigma \sqrt{n}} \leq y \big\} }
   \Big]
 - \nu(\varphi) \Phi(y) \right| \leq   \frac{C \log n}{\sqrt{n}} \| \varphi \|_{\gamma}.     \label{BerryE01_Weak}  
\end{align}
\item[\rm{(2)}] For any compact set $K \subset (\mathbb{S}_+^{d-1})^{\circ}$, 
there exists a constant $C>0$ such that for all $n \geq 1$ and $\varphi \in \mathcal{B}_{\gamma}$, 
\begin{align}
\qquad \sup_{y \in \mathbb{R}} \sup_{ x \in K }
 \left| \mathbb{E} \Big[ \varphi(X_n^x)
  \mathbbm{1}_{ \big\{ \frac{\log \rho(G_n) - n \lambda }{\sigma \sqrt{n}} \leq y \big\} }
   \Big]
   - \nu(\varphi) \Phi(y) \right| \leq  \frac{C \log n}{\sqrt{n}}  \| \varphi \|_{\gamma}.  \label{BerryE02_Weak}
\end{align}
\end{itemize} 
\end{theorem}

The proof of \eqref{BerryE01_Weak} and \eqref{BerryE02_Weak} relies on
the H\"{o}lder regularity of the stationary measure $\nu$ 
established in Proposition \ref{PropRegularity}. 
To prove Proposition \ref{PropRegularity}, the large deviation bounds 
for the norm cocycle $\log |G_n x|$ (see Theorem \ref{Thm_BRLD_changedMea}) is required.
This explains why the non-arithmeticity condition \ref{Condi-NonArith} 
is assumed in Theorem \ref{Thm-BerryE-Posi_Weak}.

It seems to be a challenging problem to improve \eqref{BerryE01_Weak} and \eqref{BerryE02_Weak}
by replacing $\frac{\log n}{\sqrt{n}}$ with $\frac{1}{\sqrt{n}}$. 

\subsection{Moderate deviation expansions}
In this section we formulate the moderate deviation results  
for the matrix norm $\| G_n \|$, the entries $G_n^{i,j}$ and the spectral radius $\rho(G_n)$. 
We need some additional notation. For any $s \in (-\eta, \eta)$,
define the transfer operator $P_s$ as follows:  for any $\varphi \in \mathcal{C}(\mathbb{S}_+^{d-1})$, 
\begin{align}\label{Def_Trans_Oper_01}
P_s \varphi (x) = \int_{\Gamma_{\mu}} |g x|^s \varphi (g \cdot x) \mu(dg), 
\quad  x \in \mathbb{S}_+^{d-1}. 
\end{align}
We see that $P_0$ coincides with the transfer operator $P$ defined by \eqref{Ch5_Def_P_qwe}.  
Based on the perturbation theory for linear operators \cite{HH01}, 
it was shown in \cite{XGL19b} that under conditions \ref{Condi-MomentH} and \ref{CondiAPH}, 
the transfer operator $P_s$ has spectral gap properties on the Banach space $\mathcal{B}_{\gamma}$ 
and possesses a dominating eigenvalue $\kappa(s)$. 
Moreover, the function $\kappa$ is analytic, real-valued and strictly convex in a small neighborhood of $0$ 
under the additional condition \ref{Condi-VarianceH}. 
Denote $\Lambda = \log \kappa$ and $\gamma_k = \Lambda^{(k)}(0)$, $k \geq 1$, 
then it holds that $\gamma_1 =\lambda$ and $\gamma_2 = \sigma^2$. 
Throughout this paper, we write $\zeta$ for the Cram\'{e}r series of $\Lambda$: 
\begin{align}\label{Def-CramSeri}
\zeta(t)=\frac{\gamma_3}{6\gamma_2^{3/2} } + \frac{\gamma_4\gamma_2-3\gamma_3^2}{24\gamma_2^3}t
+ \frac{\gamma_5\gamma_2^2-10\gamma_4\gamma_3\gamma_2 + 15\gamma_3^3}{120\gamma_2^{9/2}}t^2 + \cdots,  
\end{align}
which converges for  $|t|$ small enough. 

The following result concerns the Cram\'{e}r type moderate deviations for the operator norm $\| G_n \|$. 
Recall that $(\mathbb{S}_+^{d-1})^{\circ} = \{x > 0 : |x|=1 \}$. 

\begin{theorem}\label{Thm-Cram-Norm}
Assume conditions \ref{Condi-MomentH}, \ref{CondiAPH} and \ref{Condi-VarianceH}.
Then, for any compact set $K \subset (\mathbb{S}_+^{d-1})^{\circ}$, 
we have, as $n \to \infty$, uniformly in $x \in K$, $y \in [0, o(\sqrt{n} )]$ and $\varphi \in \mathcal{B}_{\gamma}$, 
\begin{align} 
&   \frac{\mathbb{E}
  \left[ \varphi(X_n^x) \mathbbm{1}_{ \{ \log \| G_n \| - n \lambda \geq \sqrt{n} \sigma y \} } \right] }
  { 1-\Phi(y) }
   = e^{ \frac{y^3}{\sqrt{n}}\zeta (\frac{y}{\sqrt{n}} ) }
    \left[ \nu(\varphi) +  \|\varphi \|_{\gamma}  O \left( \frac{y+1}{\sqrt{n}} \right) \right],  \label{CramerNorm01} \\
&   \frac{\mathbb{E}
  \left[ \varphi(X_n^x) \mathbbm{1}_{ \{ \log \| G_n \| - n \lambda \leq - \sqrt{n} \sigma y \} } \right] }
  { \Phi(-y) }
    =  e^{ - \frac{y^3}{\sqrt{n}}\zeta (-\frac{y}{\sqrt{n}} ) }
   \left[ \nu(\varphi) +  \|\varphi \|_{\gamma} O \left( \frac{ y+1 }{ \sqrt{n} } \right) \right].  \label{CramerNorm02}
\end{align}
\end{theorem}

 Like in Theorem \ref{Thm-BerryE-Norm},  it can also be checked that in Theorem \ref{Thm-Cram-Norm} 
 the operator norm $\|\cdot\|$ can be replaced by any matrix norm.

Note that condition \ref{Condi-KestenH} is not required in Theorem \ref{Thm-Cram-Norm}. 
Theorem \ref{Thm-Cram-Norm} is new even for $\varphi = \mathbf{1}$ 
and the expansions \eqref{CramerNorm01} and \eqref{CramerNorm02} remain valid even when $\nu(\varphi) = 0$. 
As a particular case, 
Theorem \ref{Thm-Cram-Norm} implies the following moderate deviation principle for $\log \| G_n \|$
with a target function $\varphi$ on the Markov chain $X_n^x$.  
\begin{corollary} \label{Ch5_coro-mdp-norm}
Assume conditions \ref{Condi-MomentH}, \ref{CondiAPH} and \ref{Condi-VarianceH}.
Then, for any real-valued function $\varphi \in \mathcal{B}_{\gamma}$ satisfying $\nu(\varphi) > 0$, 
for any Borel set $B \subseteq \mathbb{R}$ and any positive sequence $(b_n)_{n\geq 1}$ satisfying
$\frac{b_n}{n}\rightarrow0$ and $\frac{b_n}{\sqrt{n}}\to \infty$, 
we have, uniformly in $x \in K$, 
\begin{align} \label{Ch5_MDP_Norm}
- \inf_{y\in B^{\circ}} \frac{y^2}{2\sigma^2} & \leq   
\liminf_{n\to \infty} \frac{n}{b_n^{2}}
\log  \mathbb{E} \left[  \varphi(X_n^x)  
\mathbbm{1}_{ \big\{ \frac{\log \|G_n \| - n\lambda }{b_n} \in B  \big\} }  \right]
\nonumber\\
&  \leq   \limsup_{n\to \infty}\frac{n}{b_n^{2}}
\log  \mathbb{E}  \left[   \varphi(X_n^x) 
\mathbbm{1}_{ \big\{ \frac{\log \|G_n \| - n\lambda }{b_n} \in B  \big\} }   \right] 
\leq - \inf_{y\in \bar{B}} \frac{y^2}{2\sigma^2},
\end{align}
where $B^{\circ}$ and $\bar{B}$ are respectively the interior and the closure of $B$.
\end{corollary}
Note that the target function $\varphi$ in \eqref{Ch5_MDP_Norm} is not necessarily positive and it can vanish 
on some part of the projective space $\mathbb{S}_+^{d-1}$. 
The moderate deviation principle \eqref{Ch5_MDP_Norm} is new, even for $\varphi = \mathbf{1}$. 

As in Theorem \ref{Thm-BerryE-Norm}, it would be interesting to prove that Theorem \ref{Thm-Cram-Norm}
holds uniformly in $x \in \mathbb{S}_+^{d-1}$ instead of $x \in K$. 


Now we formulate Cram\'{e}r type moderate deviation expansions for the scalar product $\langle f, G_nx \rangle$
as well as for the spectral radius $\rho(G_n)$. 

\begin{theorem}\label{Thm-Cram-Posi-tag}
Assume conditions \ref{Condi-MomentH}, \ref{Condi-KestenH} and \ref{Condi-VarianceH}.
Then, 
we have:  \\ 
\rm{(1)} as $n \to \infty$, uniformly in $f, x \in \mathbb{S}_+^{d-1}$, $y \in [0, o(\sqrt{n} )]$ 
and $\varphi \in \mathcal{B}_{\gamma}$, 
\begin{align}
  \frac{\mathbb{E}
\left[ \varphi(X_n^x) \mathbbm{1}_{ \{ \log \langle f, G_nx \rangle - n \lambda \geq \sqrt{n} \sigma y \} } \right] }
{ 1-\Phi(y) }
&  = e^{ \frac{y^3}{\sqrt{n}}\zeta (\frac{y}{\sqrt{n}} ) }
\left[ \nu(\varphi) +  \|\varphi \|_{\gamma}  O \Big( \frac{y+1}{\sqrt{n}} \Big) \right],  \label{CramerThm01}  \\
\frac{\mathbb{E}
\left[ \varphi(X_n^x) \mathbbm{1}_{ \{ \log \langle f, G_nx \rangle - n \lambda \leq - \sqrt{n} \sigma y \} } \right] }
{ \Phi(-y) }
&  = e^{ - \frac{y^3}{\sqrt{n}}\zeta (-\frac{y}{\sqrt{n}} ) }
\left[ \nu(\varphi) +  \|\varphi \|_{\gamma} O \Big( \frac{ y+1 }{ \sqrt{n} } \Big) \right];  \label{CramerThm02}
\end{align}
\rm{(2)} for any compact set $K \subset (\mathbb{S}_+^{d-1})^{\circ}$, as $n \to \infty$, 
 uniformly in $x \in K$, $y \in [0, o(\sqrt{n} )]$ and $\varphi \in \mathcal{B}_{\gamma}$, 
\begin{align}
\frac{\mathbb{E}
\left[ \varphi(X_n^x) \mathbbm{1}_{ \{ \log \rho(G_n) - n \lambda \geq \sqrt{n} \sigma y \} } \right] }
{ 1-\Phi(y) }
&  = e^{ \frac{y^3}{\sqrt{n}}\zeta (\frac{y}{\sqrt{n}} ) }
     \left[ \nu(\varphi) +   \|\varphi \|_{\gamma} O \Big( \frac{y+1}{\sqrt{n}} \Big) \right],  \label{CramerThm03}  \\
\frac{\mathbb{E}
\left[ \varphi(X_n^x) \mathbbm{1}_{ \{ \log \rho(G_n) - n \lambda \leq - \sqrt{n} \sigma y \} } \right] }
{ \Phi(-y) }
&  = e^{ - \frac{y^3}{\sqrt{n}}\zeta (-\frac{y}{\sqrt{n}} ) }
 \left[ \nu(\varphi) +  \|\varphi \|_{\gamma} O \Big( \frac{ y+1 }{ \sqrt{n} } \Big) \right].   \label{CramerThm04}
\end{align}
\end{theorem}

As a particular case of
\eqref{CramerThm01} and \eqref{CramerThm02} with $f = e_i$ and $x = e_j$,
we get the Cram\'{e}r type moderate deviation expansions for the entries $G_n^{i,j}$. 
The expansions \eqref{CramerThm01}-\eqref{CramerThm04} are all new even for $\varphi = \mathbf{1}$.

We end this subsection by giving moderate deviation principles 
for the scalar product $\langle f, G_nx \rangle$ and for the spectral radius $\rho(G_n)$. 
Recall that for a Borel set $B$,  we write respectively $B^{\circ}$ and $\bar{B}$ for its interior and closure. 
 
\begin{theorem} \label{ciro-LDP002}
Assume either conditions \ref{Condi-MomentH}, \ref{Condi-KestenH}, \ref{Condi-VarianceH}, 
or conditions \ref{Condi-MomentH}, \ref{Condi-WeakKest}, \ref{Condi-NonArith}.
Then, for any real-valued function $\varphi \in \mathcal{B}_{\gamma}$ satisfying $\nu(\varphi) > 0$, 
for any Borel set $B \subseteq \mathbb{R}$ and any positive sequence $(b_n)_{n\geq 1}$ satisfying
$\frac{b_n}{n}\rightarrow0$ and $\frac{b_n}{\sqrt{n}}\to \infty$, 
we have
\begin{enumerate}
\item[\rm{(1)}]
uniformly in $f, x \in \mathbb{S}_+^{d-1}$,
\begin{align} \label{MDP_Entry_nn}
 -\inf_{y\in B^{\circ}} \frac{y^2}{2\sigma^2} & \leq 
\liminf_{n\to \infty} \frac{n}{b_n^{2}}
\log  
\mathbb{E} \left[  \varphi(X_n^x)  
\mathbbm{1}_{ \left\{ \frac{\log \langle f, G_n x \rangle - n\lambda }{b_n} \in B  \right\} }  \right]
\nonumber\\
 &  \leq  \limsup_{n\to \infty}\frac{n}{b_n^{2}}
\log   
\mathbb{E}  \left[   \varphi(X_n^x) 
\mathbbm{1}_{  \left\{ \frac{\log \langle f, G_n x \rangle - n\lambda }{b_n} \in B  \right\} }   \right] 
\leq - \inf_{y\in \bar{B}} \frac{y^2}{2\sigma^2}; 
\end{align}

\item[\rm{(2)}]
for any compact set $K \subset (\mathbb{S}_+^{d-1})^{\circ}$, uniformly in $x \in K$, 
\begin{align} \label{MDP_Spe_Radius_nn}
 -\inf_{y\in B^{\circ}} \frac{y^2}{2\sigma^2}  & \leq    
\liminf_{n\to \infty} \frac{n}{b_n^{2}}
\log  
\mathbb{E} \left[  \varphi(X_n^x)  
\mathbbm{1}_{ \left\{ \frac{\log \rho(G_n) - n\lambda }{b_n} \in B  \right\} }  \right]
\nonumber\\
&  \leq  \limsup_{n\to \infty}\frac{n}{b_n^{2}}
\log   
\mathbb{E}  \left[   \varphi(X_n^x) 
\mathbbm{1}_{  \left\{ \frac{\log \rho(G_n) - n\lambda }{b_n} \in B  \right\} }   \right] 
\leq - \inf_{y\in \bar{B}} \frac{y^2}{2\sigma^2}. 
\end{align}
\end{enumerate}
\end{theorem}

Under conditions \ref{Condi-MomentH}, \ref{Condi-KestenH} and \ref{Condi-VarianceH}, 
the moderate deviation principles \eqref{MDP_Entry_nn} and \eqref{MDP_Spe_Radius_nn} 
follow directly from Theorem \ref{Thm-Cram-Posi-tag}, 
just as we obtained \eqref{Ch5_MDP_Norm} from Theorem \ref{Thm-Cram-Norm}. 
Under conditions \ref{Condi-MomentH}, \ref{Condi-WeakKest}, \ref{Condi-NonArith}, 
\eqref{MDP_Entry_nn} and \eqref{MDP_Spe_Radius_nn} cannot be deduced from 
Theorem \ref{Thm-Cram-Posi-tag}. 
In fact, the proof turns out to be delicate and is carried out using the H\"{o}lder regularity
of the stationary measure $\nu$, see Proposition \ref{PropRegularity}.

\subsection{Formulas for the asymptotic variance}
In this section, we give alternative expresssions for the asymptotic variance $\sigma^2$ defined by \eqref{Formu_sig}.  
These expressions can be useful while applying the theorems and the corollaries stated before,
 where $\sigma$ appears.  

\begin{proposition}\label{Prop_Variance}
(1) Under conditions \ref{Condi-MomentH} and \ref{CondiAPH}, we have 
\begin{align} \label{exp-sigma2-1}
\sigma^2 = \lim_{n\to\infty} \frac{1}{n} \mathbb{E} \left[ \big( \log \|G_n \| - n \lambda \big)^2 \right].
\end{align}
(2) Under conditions \ref{Condi-MomentH} and \ref{Condi-KestenH}, we have 
\begin{align} \label{exp-sigma2-2}
\sigma^2 =   \lim_{n\to\infty} \frac{1}{n} 
            \mathbb{E} \left[ \big( \log \langle f, G_n x \rangle - n \lambda  \big)^2 \right]
         =   \lim_{n\to\infty} \frac{1}{n} \mathbb{E} \left[ \big( \log \rho(G_n) - n \lambda \big)^2 \right],
\end{align}
where the convergence in the first equality holds uniformly in $f, x \in \mathbb{S}_+^{d-1}$. 
\end{proposition}

We mention that for invertible matrices,  the expression \eqref{exp-sigma2-1} 
has been established in \cite[Proposition 14.7]{BQ16b}.  
For positive matrices, both  \eqref{exp-sigma2-1} and  \eqref{exp-sigma2-2} are new.

\section{Proofs of Berry-Esseen bounds}

The goal of this section is to establish
Theorems \ref{Thm-BerryE-Norm}, \ref{Thm-BerryE-Posi} and \ref{Thm-BerryE-Posi_Weak}. 

\subsection{Proof of Theorems \ref{Thm-BerryE-Norm} and \ref{Thm-BerryE-Posi}}

In order to prove Theorem \ref{Thm-BerryE-Norm}, we shall use the following result which was shown in 
\cite[Lemma 4.5]{BDGM14}. 
\begin{lemma} \label{Lem_Norm}
Under condition \ref{CondiAPH} (i), for any $x \in (\mathbb{S}_+^{d-1})^{\circ}$, we have 
\begin{align*}
\tau(x) := \inf_{g \in \Gamma_{\mu} }  \frac{|gx|}{ \|g\| } > 0. 
\end{align*}
Moreover, for any compact set $K \subset (\mathbb{S}_+^{d-1})^{\circ}$, 
it holds that $\inf_{x \in K} \tau(x) > 0$. 
\end{lemma}

We  now proceed to prove Theorem \ref{Thm-BerryE-Norm} based on Lemma \ref{Lem_Norm}
and the Berry-Esseen bound for the norm cocycle $\log |G_n x|$ established in \cite{XGL19b}. 

\begin{proof}[Proof of Theorem \ref{Thm-BerryE-Norm}]
Without loss of generality, we assume that the target function $\varphi$ is non-negative. 
Under conditions of Theorem \ref{Thm-BerryE-Norm}, the following Berry-Esseen bound for the norm cocycle $\log |G_n x|$
with a target function $\varphi$ on the Markov chain $X_n^x$
has been recently established in \cite{XGL19b}: 
there exists a constant $C>0$ such that for all $n \geq 1$ and $\varphi \in \mathcal{B}_{\gamma}$, 
\begin{align}\label{Norm-BerryE-Posi}
\sup_{y \in \mathbb{R}} \sup_{x \in \mathbb{S}_+^{d-1} }
\left| \mathbb{E} \Big[ \varphi(X_n^x)
\mathbbm{1}_{ \big\{ \frac{\log | G_n x | - n \lambda }{ \sigma \sqrt{n} } \leq y \big\} }
   \Big]
 - \nu(\varphi) \Phi(y) \right| \leq \frac{C}{\sqrt{n}}  \| \varphi \|_{\gamma}.
\end{align}
On the one hand, using the fact that $\log |G_n x| \leq \log \| G_n \|$,
we deduce from \eqref{Norm-BerryE-Posi} that
there exists a constant $C>0$ such that for all 
$y \in \mathbb{R}$, $x \in \mathbb{S}_+^{d-1}$, $n \geq 1$ and $\varphi \in \mathcal{B}_{\gamma}$, 
\begin{align*}
\mathbb{E} \left[ \varphi(X_n^x)
\mathbbm{1}_{ \big\{ \frac{\log \| G_n \| - n \lambda }{ \sigma \sqrt{n} } \leq y \big\} }
\right]
\leq \nu(\varphi) \Phi(y) + \frac{C}{\sqrt{n}} \| \varphi \|_{\gamma}.
\end{align*}
On the other hand, by Lemma \ref{Lem_Norm}, 
we see that for any compact set $K \subset (\mathbb{S}_+^{d-1})^{\circ}$,
there exists a constant $C_1>0$ such that for all $n \geq 1$ and $x \in K$,
\begin{align} \label{Posi-NormLower01}
\log \| G_n \| \leq \log |G_n x| + C_1.
\end{align}
Combining this inequality with \eqref{Norm-BerryE-Posi}, we obtain that, with $y_1 = y - \frac{C_1}{\sigma \sqrt{n}}$,
uniformly in $y \in \mathbb{R}$, $x \in K$, $n \geq 1$ and $\varphi \in \mathcal{B}_{\gamma}$, 
\begin{align}\label{Pf_BE_Norm_vvv}
\mathbb{E} \left[ \varphi(X_n^x)
\mathbbm{1}_{ \big\{ \frac{\log \| G_n \| - n \lambda }{ \sigma \sqrt{n} } \leq y \big\} }
\right]
\geq \nu(\varphi) \Phi(y_1) - \frac{C}{\sqrt{n}}  \| \varphi \|_{\gamma}.
\end{align}
By elementary calculations, we find that there exists a constant $C_2 >0$ such that
for all $y \in \mathbb{R}$ and $n \geq 1$,
\begin{align}\label{BE_Calcu_kkk}
\Phi(y_1) - \Phi(y) = - \frac{1}{\sqrt{2\pi}} \int_{y - \frac{C_1}{\sigma \sqrt{n}}}^{y} e^{-\frac{t^2}{2}} dt 
\geq -\frac{C_2}{\sqrt{n}}.
\end{align}
This, together with \eqref{Pf_BE_Norm_vvv}, yields the desired lower bound.
The proof of Theorem \ref{Thm-BerryE-Norm} is complete.
\end{proof}

Now we proceed to prove Theorem \ref{Thm-BerryE-Posi}. 
For any $0< \epsilon <1$, set
\begin{align*}
\mathbb{S}_{+, \epsilon}^{d-1}
= \left\{ x \in \mathbb{S}_+^{d-1}:
\langle x, e_j \rangle \geq \epsilon \ \mbox{for all $1\leq j\leq d$}  \right\}.
\end{align*}
The following result provides an equivalent formulation of condition \ref{Condi-KestenH},
which will be used to prove Theorems \ref{Thm-BerryE-Posi} and \ref{Thm-Cram-Posi-tag}. 
For any matrix $g\in \supp \mu$, 
we denote $g \cdot \mathbb{S}_{+}^{d-1} = \big\{ g \cdot x: x \in \mathbb{S}_{+}^{d-1} \big\}$. 

\begin{lemma} \label{lem equiv Kesten}
Condition \ref{Condi-KestenH} is equivalent to the following statement:
there exists a constant $\epsilon \in (0,\frac{\sqrt{2}}{2})$ such that
\begin{align} \label{equicon A4}
g \cdot \mathbb{S}_{+}^{d-1} \subset \mathbb{S}_{+, \epsilon}^{d-1}, 
\quad \mbox{for any} \ g\in \supp \mu.
\end{align}
\end{lemma}

\begin{proof} 
We first show that the assertion \eqref{equicon A4} implies condition \ref{Condi-KestenH}.
For any matrix $g = ( g^{i,j} )_{ 1\leq i, j \leq d } \in \supp \mu$, 
we see that for any $1\leq i, j  \leq d$, 
\begin{align} \label{equilem 01}
\langle e_i, g \cdot e_j \rangle
= \frac{g^{i,j}}{ \sqrt{\sum_{i=1}^d (g^{i,j})^2}}.
\end{align}
Using \eqref{equicon A4} and the definition of $\mathbb{S}_{+,\epsilon}^{d-1}$, 
we get that there exists $\epsilon \in (0,\frac{\sqrt{2}}{2})$ such that 
$\langle e_i, g\!\cdot\! e_j \rangle \geq \epsilon$ for all $1\leq i,j \leq d$. 
This implies condition \ref{Condi-KestenH} with $C = \sqrt{ \frac{1}{d-1} ( \frac{1}{\epsilon^2} -1 ) } $
by taking maxima and minima by rows in \eqref{equilem 01}. 

We next prove that condition \ref{Condi-KestenH} implies the assertion \eqref{equicon A4}. 
For any $x \in \mathbb{S}_+^{d-1}$, we write $x=\sum_{j=1}^d x_j e_j$, 
where $x_j\geq 0$ satisfies $\sum_{j=1}^d x_j^2 =1$. 
It is easy to see that $\sum_{j=1}^d x_j \geq 1$.  
For any $1 \leq i \leq d$, it holds that
\begin{align*}
\langle e_i, g\!\cdot\! x \rangle 
= \frac{1}{|gx|} \sum_{j=1}^d x_j \langle e_i, g e_j \rangle  
=  \frac{ \sum_{j=1}^d x_j g^{i,j}}{ \sqrt{ \sum_{i=1}^d ( \sum_{j=1}^d g^{i,j} x_j )^2 } }.
\end{align*}
Since $\sum_{j=1}^d x_j^2 =1$, we get $(\sum_{j=1}^d g^{i,j}x_j )^2 \leq \sum_{j=1}^d (g^{i,j})^2$
using the Cauchy-Schwarz inequality.
Combining this with condition \ref{Condi-KestenH} and the fact that $\sum_{j=1}^d x_j \geq 1$,
we obtain $\langle e_i, g\!\cdot\! x \rangle \geq \sum_{j=1}^d \frac{x_j }{\sqrt{ C^2 d^2} }
\geq \frac{1}{Cd}$,
 so that the assertion \eqref{equicon A4} holds with $\epsilon =  \frac{1}{Cd}.$
\end{proof}

Using Lemma \ref{lem equiv Kesten}, 
Theorem \ref{Thm-BerryE-Norm} and the Berry-Esseen bound \eqref{Norm-BerryE-Posi}, 
we are in a position to prove Theorem \ref{Thm-BerryE-Posi}.

\begin{proof}[Proof of Theorem \ref{Thm-BerryE-Posi}]
Without loss of generality, we assume that the target function $\varphi$ is non-negative.

We first prove the Berry-Esseen bound \eqref{BerryE01} for the scalar product $ \langle f, G_n x \rangle$. 
On the one hand, using the fact that $\log \langle f, G_n x \rangle \leq \log |G_n x|$,
we deduce from the Berry-Esseen bound \eqref{Norm-BerryE-Posi} that
there exists a constant $C>0$ such that for all
$y \in \mathbb{R}$, $f, x \in \mathbb{S}_+^{d-1}$, $n \geq 1$ and $\varphi \in \mathcal{B}_{\gamma}$,
\begin{align}\label{BE_Low_fff}
\mathbb{E} \left[ \varphi(X_n^x)
\mathbbm{1}_{ \big\{ \frac{\log \langle f, G_n x \rangle - n \lambda }{ \sigma \sqrt{n} } \leq y \big\} }
\right]
\geq \nu(\varphi) \Phi(y) -\frac{C}{\sqrt{n}} \| \varphi \|_{\gamma}.
\end{align}
On the other hand, note that $\log \langle f, G_n x \rangle = \log |G_n x| + \log \langle f, X_n^x \rangle$. 
By Lemma \ref{lem equiv Kesten}, we see that
there exists a constant $C_1>0$ such that for all $f, x \in \mathbb{S}_+^{d-1}$ and $n \geq 1$,
\begin{align} \label{Posi-ScalLower01}
\log | G_n x| \leq \log \langle f, G_n x \rangle + C_1.
\end{align}
Using this inequality and again the Berry-Esseen bound \eqref{Norm-BerryE-Posi}, 
we obtain that, with $y_1 = y + \frac{C_1}{\sigma \sqrt{n}}$, uniformly in 
$y \in \mathbb{R}$, $f, x \in \mathbb{S}_+^{d-1}$, $n \geq 1$ and $\varphi \in \mathcal{B}_{\gamma}$,
\begin{align*}
\mathbb{E} \left[ \varphi(X_n^x)
\mathbbm{1}_{ \big\{ \frac{\log \langle f, G_n x \rangle - n \lambda }{ \sigma \sqrt{n} } \leq y \big\} }
\right]
\leq \nu(\varphi) \Phi(y_1) + \frac{C}{\sqrt{n}}  \| \varphi \|_{\gamma}.
\end{align*}
It is easy to show that $\Phi(y_1) - \Phi(y) \leq \frac{C}{\sqrt{n}}$, uniformly in $y \in \mathbb{R}$.
Together with the above inequality, this leads to the desired upper bound
and ends the proof of the Berry-Esseen bound \eqref{BerryE01}.

We next prove the bound \eqref{BerryE02} for the spectral radius $\rho(G_n)$. 
Since $\rho(G_n) \leq \| G_n \|$, by Theorem \ref{Thm-BerryE-Norm}, 
we get the following lower bound: 
there exists a constant $C>0$ such that for all 
$y \in \mathbb{R}$, $x \in K$, $n \geq 1$ and $\varphi \in \mathcal{B}_{\gamma}$,
\begin{align*}
\mathbb{E} \left[ \varphi(X_n^x)
\mathbbm{1}_{ \big\{ \frac{\log \rho(G_n) - n \lambda }{ \sigma \sqrt{n} } \leq y \big\} }
\right]
\geq \nu(\varphi) \Phi(y) -\frac{C}{\sqrt{n}} \| \varphi \|_{\gamma}.
\end{align*}
The upper bound is carried out by using the Collatz-Wielandt formula 
in conjugation with the Berry-Esseen bound \eqref{BerryE01} for the entries $G_n^{i,i}$. 
Denote by $\mathcal{C}_+ = \{ x \in \mathbb{R}^d: x \geq 0 \} \setminus \{0\}$ 
 the positive quadrant  in $\mathbb{R}^d$ except the origin. 
According to the Collatz-Wielandt formula, 
the spectral radius of the positive matrix $G_n$ can be represented as follows:
\begin{align}\label{Formu_ColWie}
\rho( G_n ) = \sup_{ x \in \mathcal{C}_+ } 
   \min_{ 1 \leq i \leq d, \langle  e_i, x \rangle > 0 } \frac{ \langle e_i, G_n x \rangle }{ \langle e_i, x \rangle}. 
\end{align}
It follows that there exists a constant $\epsilon \in (0,\frac{\sqrt{2}}{2})$ 
such that for all $x \in \mathbb{S}_+^{d-1}$,
\begin{align}\label{Ine_Spectral01}
\rho( G_n ) \geq \min_{ 1 \leq i \leq d }  \langle e_i, G_n x \rangle 
\geq   \min_{ 1 \leq i \leq d }  \langle e_i, X_n^x \rangle |G_n x|
\geq  \epsilon |G_n x|,
\end{align}
where in the last inequality we used Lemma \ref{lem equiv Kesten}. 
Using the bound \eqref{Norm-BerryE-Posi} and the inequality \eqref{Ine_Spectral01}, we deduce that 
there exists a constant $C>0$ such that 
for all $x \in \mathbb{S}_+^{d-1}$, $y \in \mathbb{R}$, $n \geq 1$ and $\varphi \in \mathcal{B}_{\gamma}$,
\begin{align*}
\mathbb{E} \left[ \varphi(X_n^{x})
\mathbbm{1}_{ \big\{ \frac{\log \rho(G_n) - n \lambda }{ \sigma \sqrt{n} } \leq y \big\} }
\right]
\leq \nu(\varphi) \Phi(y) + \frac{C}{\sqrt{n}} \| \varphi \|_{\gamma}.
\end{align*}
This ends the proof of the bound \eqref{BerryE02} for the spectral radius $\rho(G_n)$. 
\end{proof}

\subsection{H\"{o}lder regularity of stationary measures}\label{subsec_HolderRegu}

In this section we present our results on the H\"{o}lder regularity 
of the stationary measure $\pi_s$ and of the eigenmeasure $\nu_s$.
The regularity of $\pi_s$ and $\nu_s$  is central to establishing 
Berry-Esseen type bounds and moderate deviation asymptotics for the entries $G_n^{i,j}$
and is also of independent interest. 
Hereafter, we denote $$I_{\mu} = \{ s \geq 0: \bb E(\|g_1\|^s) < \infty \}.$$ 
By H\"{o}lder's inequality, it is easy to see that $I_{\mu}$ is an interval on $\bb R$. 
The interior of $I_{\mu}$ is denoted by $I_{\mu}^{\circ}$. 
For any $s\in I_{\mu}$, define the transfer operator $P_s$ 
as in \eqref{Def_Trans_Oper_01}: for any $\varphi \in \mathcal{B}_{\gamma}$, 
\begin{align}\label{transfoper001}
P_{s}\varphi(x) = \int_{\Gamma_{\mu}}  |g x|^s \varphi( g \cdot x ) \mu(dg), 
\quad  x\in \bb S_+^{d-1}.  
\end{align}
It is proved in \cite{BDGM14} that the operator $P_s$ has unique continuous strictly positive eigenfunction $r_s$
on $\mathbb{P}^{d-1}$ and unique probability eigenmeasure $\nu_s$
satisfying 
\begin{align*}
P_s r_s = \kappa(s) r_s  \quad \mbox{and} \quad  P_s \nu_s = \kappa(s) \nu_s. 
\end{align*}
The family of probability kernels  
$q_{n}^{s}(x,g) = \frac{ |gx|^s }{\kappa^{n}(s)} \frac{r_{s}(g \cdot x)}{r_{s}(x)},$
$n\geq 1$, satisfies the cocycle property. 
Hence the probability measures 
$q_{n}^{s}(x,g_{n}\dots g_{1})\mu(dg_1)\dots\mu(dg_n)$
form a projective system on $\mathcal M_+^{\bb N^*}$, 
so that there exists a unique probability measure  $\mathbb Q_s^x$ on $\mathcal M_+^{\bb N^*}$, 
by the Kolmogorov extension theorem.  
The corresponding expectation is denote by $\mathbb{E}_{\mathbb Q_s^x}$.
For any measurable function $\varphi$ on $(\bb S_+^{d-1} \times \mathbb R)^{n}$, 
it holds that 
\begin{align}\label{basic equ1}
 \frac{1}{ \kappa^{n}(s) r_{s}(x) }
\mathbb{E} \Big[  r_{s}(X_{n}^{x}) & |G_nx|^{s}  \varphi \big( X_{1}^{x}, \log |G_1 x|,\dots, X_{n}^{x}, \log |G_n x| 
                    \big) \Big]   \nonumber\\
&\quad 
=\mathbb{E}_{\mathbb{Q}_{s}^{x}} \Big[ \varphi \big( X_{1}^{x}, \log |G_1 x|,\dots, X_{n}^{x}, \log |G_n x|  \big) \Big].
\end{align}
Under the changed measure  $\mathbb Q_s^x$, the Markov chain $(X_n^x)_{n \geq 0}$ has 
a unique stationary measure $\pi_s$ given by $\pi_s(\varphi) = \frac{\nu_{s}(\varphi r_{s})}{\nu_{s}(r_{s})}$, 
for any function $\varphi \in \mathcal{C}(\bb S_+^{d-1})$.

\begin{proposition}\label{PropRegularity}
Assume either condition \ref{Condi-KestenH} 
or conditions \ref{Condi-MomentH}, \ref{Condi-WeakKest}, \ref{Condi-NonArith}. 
Then, for any $s \in \{0\} \cup I_{\mu}^{\circ}$, there exists a constant $\alpha > 0$ such that 
\begin{align} \label{RegularityIne00}
\sup_{f \in \bb S_+^{d-1} } \int_{ \bb S_+^{d-1} }  \frac{1}{|\langle f,x \rangle|^{\alpha}} \nu_s(dx) < +\infty. 
\end{align}
In particular, for any $s \in \{0\} \cup I_{\mu}^{\circ}$,
there exist constants $\alpha, C >0$ such that for any $0< t < 1$, 
\begin{align} \label{RegularityIne}
\sup_{f \in \bb S_+^{d-1} }  \nu_s \left( \left\{x \in \bb S_+^{d-1}:  |\langle f, x \rangle| \leq t \right\}  \right)  
\leq C t^{\alpha}. 
\end{align}
Moreover, the assertions \eqref{RegularityIne00} and \eqref{RegularityIne} remain valid
when the eigenmeasure $\nu_s$ is replaced by the stationary measure $\pi_s$. 
\end{proposition}

Under condition \ref{Condi-KestenH},
the proof of the assertion \eqref{RegularityIne00} relies on the fact that 
$\supp \nu = \supp \nu_s$ ($s > 0$) established in \cite{BDGM14}
and essentially on condition \ref{Condi-KestenH} which ensures that the Markov chain $(X_n^x)_{n \geq 0}$
stays forever in the interior of the projective space $\mathbb{S}_+^{d-1}$: see Lemma \ref{lem equiv Kesten}. 
If condition \ref{Condi-KestenH} is replaced by \ref{Condi-WeakKest}, 
the main difficulty to prove \eqref{RegularityIne00} is that the Markov chain $(X_n^x)_{n \geq 0}$
is no longer separated from the coordinates $(e_k)_{1 \leq k \leq d}$, hence the proof can not follow directly from
the fact that $\supp \nu = \supp \nu_s$.  
Instead, the main ingredient in our proof consists in 
 the large deviation asymptotic for the norm cocycle $\log |G_n x|$ under the changed measure
 $\mathbb{Q}_s^x$ established in Theorem \ref{Thm_BRLD_changedMea}. 
 
It is worth mentioning that in the case of invertible matrices, 
the corresponding result with $s =0$ (in this case also $\pi_0=\nu_0=\nu$)
has been obtained in \cite{GR85};     
we also refer to \cite{BL85} for the detailed description of the method used in  \cite{GR85}
and to \cite{BFLM11, BQ16b} for a different approach of the proof. 

Before proving Proposition \ref{PropRegularity}, 
let us give the precise large deviation result for the norm cocycle $\log |G_n x|$ 
under the changed measure $\mathbb Q_s^x$.
It is deduced from \cite[Theorem 2.2]{XGL19a} and will be used in the proof of regularity 
of the stationary measure $\pi_s$ (see Proposition \ref{PropRegularity}). 
As in \eqref{Def-CramSeri}, we denote $\Lambda = \log \kappa$ and
by $\Lambda^*$ the Legendre transform of $\Lambda$. 
In particular, we have $\Lambda^*(q_s) = s q_s - \Lambda(s)$ if $q_s = \Lambda'(s)$. 

\begin{theorem}  \label{Thm_BRLD_changedMea}
Assume conditions \ref{Condi-MomentH}, \ref{CondiAPH} and \ref{Condi-NonArith}.  
Let $s \in I_{\mu}$, $t \in I^{\circ}_{\mu}$ 
be such that $s<t$ and set  $q_s = \Lambda'(s)$ and $q_t = \Lambda'(t)$. 
Then, for any positive sequence $(l_n)_{n \geq 1}$ satisfying $\lim_{n \to \infty} l_n = 0$,
we have, as $n \to \infty$, uniformly in $|l|\leq l_n$ and $x \in \bb S_+^{d-1}$, 
\begin{align*} 
&  \mathbb Q_s^x  \big( \log|G_n x|  \geq n( q_t+l ) \big)   \nonumber\\
&   =  \frac{ \nu_t(r_s) }{ \nu_t(r_t) } \frac{ r_{t}(x) }{ r_{s}(x) }  
   \frac{ \exp  \{  -n(\Lambda^*(q_t + l) - \Lambda^*(q_s) - s(q_t -q_s +l)) \} }
 {(t -s)\sigma_{t} \sqrt{2\pi n}} [ 1 + o(1)].
\end{align*}
\end{theorem}

\begin{proof}
By \eqref{basic equ1}, we get
\begin{align*}
& \mathbb Q_s^x ( \log|G_n x|  \geq n(q_t + l)) \\
& =    \frac{1}{\kappa^n(s) r_s(x)}
\mathbb{E} \left[ r_{s}(X_n^x)  |G_n x|^s 
\mathbbm{1}_{\{\log|G_n x|  \geq n(q_t + l)\}} \right]  \nonumber\\
&  =    \frac{1}{\kappa^n(s) r_s(x)}  e^{sn(q_t+l)} 
 \mathbb{E} \Big[ r_{s}(X_n^x) \psi_s \big( \log |G_n x| - n(q_t+l) \big) \Big],
\end{align*}
where $\psi_s(u) = e^{su} \mathbbm{1}_{\{u \geq 0 \}}$, $u \in \mathbb{R}$.
From Theorem 2.2 in \cite{XGL19a} it follows that 
for any $t \in I^{\circ}_{\mu}$, $q_t = \Lambda'(t)$, 
$\varphi \in \mathcal{B}_{\gamma}$ and measurable function $\psi$ on $\mathbb{R}$ 
such that $u \mapsto e^{-s'u}\psi(u)$ is directly Riemann integrable for some $s' \in (0, s)$,
we have, as $n \to \infty$, uniformly in $|l|\leq l_n$ and $x \in \bb S_+^{d-1}$,
\begin{align} \label{Petrov-Target01}
& \mathbb{E} \Big[ \varphi(X_n^x) \psi(  \log|G_n x| - n( q_t + l )) \Big]  \nonumber\\
&  \quad = \frac{r_t(x)}{ \nu_t(r_t) }   
  \frac{ \exp \left( -n \Lambda^*(q_t + l) \right) }{\sigma_t \sqrt{2\pi n} }
  \left[  \nu_t(\varphi) \int_{\mathbb{R}} e^{-ty} \psi(y) dy  +  o(1)  \right]. 
\end{align}
Using \eqref{Petrov-Target01} with $\varphi = r_s$ and $\psi = \psi_s$, 
we obtain that, uniformly in $|l|\leq l_n$ and $x \in \bb S_+^{d-1}$,
\begin{align*}
 \mathbb{E} \Big[ r_{s}(X_n^x) \psi_s \big( \log |G_n x| - n(q_t+l) \big) \Big]  
 =   \frac{r_t(x)}{ \nu_t(r_t) }  \nu_{t}(r_s) \frac{e^{-n \Lambda^*(q_t+l)}}{(t -s)\sigma_{t} \sqrt{2 \pi n} } 
 \big[ 1 + o(1) \big]. 
\end{align*}
We conclude the proof of Theorem \ref{Thm_BRLD_changedMea} by using the fact that 
$\Lambda^*(q) = sq - \Lambda(s)$ and $\Lambda(s) = \log \kappa(s)$.
\end{proof}

%
%
%
%

\begin{proof}[Proof of Proposition \ref{PropRegularity}]

As mentioned before, 
we only need to establish \eqref{RegularityIne00} and \eqref{RegularityIne} for the stationary measure $\pi_s$
since $r_s$ is bounded away from infinity and $0$ uniformly on $\mathbb{S}_+^{d-1}$.

We first prove Proposition \ref{PropRegularity} under condition \ref{Condi-KestenH}. 
By Lemma \ref{lem equiv Kesten}, the Markov chain $(X_n^x)_{n\geq 0}$ stays in the space 
$\mathbb{S}_{+,\epsilon}^{d-1}$, and therefore the support of 
its stationary measure $\nu$
is included in  $\mathbb{S}_{+,\epsilon}^{d-1}$.
Since $\supp \nu_s = \supp \nu$ for $s \in I_{\mu}$ (by \cite[Proposition 3.1]{BDGM14}), it holds that
 $\supp \nu_s \subset \mathbb{S}_{+,\epsilon}^{d-1}$.
As a consequence we also have $\supp \pi_s \subset \mathbb{S}_{+,\epsilon}^{d-1}$. 
This implies that $\langle f, x \rangle \geq \epsilon$
for all $f \in \mathbb{S}_+^{d-1}$, $x \in \supp \pi_s$, 
and so the bounds \eqref{RegularityIne00} and \eqref{RegularityIne} hold under condition \ref{Condi-KestenH}.

We next prove Proposition \ref{PropRegularity}
under conditions \ref{Condi-MomentH}, \ref{Condi-WeakKest}, \ref{Condi-NonArith}. 
We divide the proof into two steps. 
It is worth mentioning that the assertions shown below remain valid when $s = 0$. 

\textit{Step 1.} We prove that there exist two constants $C_1, C_2>0$ and an integer $n_0 \geq 1$ satisfying $C_1 > \Lambda'(s)$
 such that, for any $n \geq n_0$, 
 it holds uniformly in $f,x \in \mathbb S^{d-1}_+$ that 
\begin{align}\label{Pf_RegPosi001}
I_n: = \mathbb Q_s^x \left( \langle f,  X_n^x  \rangle \leq e^{-C_1 n} \right) 
 \leq e^{- C_2 n }. 
\end{align}
Let $s \in I_{\mu}$, $t\in I^{\circ}_{\mu}$ be such that $s<t$ 
and set  $q_s=\Lambda'(s)$ and $q_t = \Lambda'(t)$ (we allow $s$ to be $0$). 
Substituting $X_n^x = \frac{G_n x}{ |G_n x| }$ into \eqref{Pf_RegPosi001} , we have
\begin{align}\label{Pf_RegPosiaa}
I_n \leq  \mathbb Q_s^x \big(  \log |G_n x| > n q_t \big)
 + \mathbb Q_s^x \Big(  \log \langle f, G_n x  \rangle  \leq - (C_1 - q_t) n \Big).  
\end{align}
Since $s < t$, by Theorem \ref{Thm_BRLD_changedMea} we get that there exists a constant $c>0$
such that 
the first term on the right-hand side of \eqref{Pf_RegPosiaa} is bounded by $e^{- c n }$, 
uniformly in $x \in \mathbb S^{d-1}_+$. 
For the second term on the right-hand side of \eqref{Pf_RegPosiaa}, 
applying the Markov inequality and the change of measure formula \eqref{basic equ1}, 
it follows that for a sufficiently small constant $c_1>0$, uniformly in $f, x \in \mathbb{S}_+^{d-1}$, 
\begin{align} \label{Pf_Regu_Posi_00}
&  \mathbb Q_s^x \Big(  \log \langle f, G_n x  \rangle  \leq - (C_1 - q_t) n \Big) \nonumber\\
& \leq  e^{ - c_1 (C_1 - q_t)n } 
   \mathbb{E}_{ \mathbb Q_s^x }  \left(  \frac{1}{ \langle f, G_n x  \rangle^{c_1} } \right)  \nonumber\\
&   =  e^{ - c_1 (C_1 - q_t)n }  
  \mathbb{E} \left(  \frac{|G_nx|^{s}}{ \kappa^{n}(s) } \frac{ r_{s}(X_{n}^{x}) }{ r_{s}(x) } 
  \frac{1}{ \langle f, G_n x  \rangle^{c_1} } \right) \nonumber\\
&  \leq  e^{ - c_1 (C_1 - q_t)n }  
  \mathbb{E} \left(  \frac{|G_nx|^{s}}{ \kappa^{n}(s) } \frac{ r_{s}(X_{n}^{x}) }{ r_{s}(x) } 
  \frac{1}{ \min_{1 \leq i, j \leq d} \langle e_i, G_n e_j \rangle^{c_1} } \right),  
\end{align}
where in the last line we used the fact that  
$\min_{1 \leq i, j \leq d} \langle e_i, g e_j \rangle = \inf_{f, x \in \mathbb{S}_+^{d-1} } \langle f, g x \rangle$ for any $g \in \Gamma_{\mu}$. 
Since $|G_nx|\leq \|G_n\|$ and the function $r_s$ is uniformly bounded and strictly positive on $\mathbb{S}_+^{d-1}$,
using the H\"{o}lder inequality leads to 
\begin{align} \label{Pf_Regu_Posi_aa}
  & \mathbb{E} \left(  \frac{|G_nx|^{s}}{ \kappa^{n}(s) } \frac{ r_{s}(X_{n}^{x}) }{ r_{s}(x) } 
   \frac{1}{ \langle e_i, G_n e_j  \rangle^{c_1} } \right)     \nonumber\\
   &  \leq   \kappa^{-n}(s)  \mathbb{E}^{ \frac{1}{p} } \big(  \| G_n \|^{sp}  \big)
          \mathbb{E}^{ \frac{1}{p'} } \left(  \frac{ 1 }{ \langle e_i, G_n e_j  \rangle^{c_1 p'}  }  \right)   \nonumber\\
   &  \leq  \kappa^{-n}(s)  \mathbb{E}^{ \frac{1}{p} } \big(  \| G_n \|^{sp}  \big)
          \mathbb{E}^{ \frac{n}{p'} } 
           \left(  \frac{ 1 }{ \min_{1\leq i, j \leq d} \langle e_i, g_1 e_j  \rangle^{c_1 p'}  }  \right), 
\end{align}
where $1/p + 1/p' = 1$ with $p, p' >1$. 
Recall that $c_1 > 0$ can be taken sufficiently small. Taking $p$ sufficiently close to $1$ ($p'$ sufficiently large)
and using condition \ref{Condi-WeakKest},
we get that  the right-hand side of \eqref{Pf_Regu_Posi_aa} is dominated by $e^{Cn}$ with some constant $C>0$. 
Consequently, in view of \eqref{Pf_Regu_Posi_00}, choosing the constant $C_1>0$ sufficiently large,
we obtain that the right-hand side of \eqref{Pf_Regu_Posi_00} is bounded by 
$e^{- C_2 n}$ with some constant $C_2 > 0$, uniformly in $f, x \in \mathbb S^{d-1}_+$.

\textit{Step 2.} 
From the construction of $\mathbb{Q}_s^x$ and the definition of $\pi_s$, 
one can verify that for any $x \in \mathbb S_+^{d-1}$ and $n \geq 1$,  
$\pi_s = (\mathbb{Q}_s^x)^{*n} * \pi_s$, where $*$ stands for the convolution of two measures. 
Combining this with \eqref{Pf_RegPosi001}, we get that for any $s \in I_{\mu}$,
uniformly in $f \in \mathbb S_+^{d-1}$, 
\begin{align} \label{Pf_RegPosi004}
  \pi_s \big(\{x:  \langle f, x \rangle \leq e^{- C_1 n} \}  \big)
= \int_{\mathbb S_+^{d-1}} (\mathbb{Q}_s^x)^{*n} 
\big( \langle f, X_n^x \rangle  \leq e^{- C_1 n} \big) \pi_s(dx) 
\leq e^{ - C_2 n},
\end{align}
where $C_1$ and $C_2$ are positive constants given in step 1. 
For $n \geq 1$, denote $B_{f,n}:= \{x \in \mathbb S_+^{d-1}:
e^{- C_1 (n+1)} \leq  \langle f, x \rangle \leq e^{ -C_1 n } \}$. 
Choosing $\alpha \in (0, C_2/C_1 )$, 
we deduce from \eqref{Pf_RegPosi004} that, uniformly in $f \in \mathbb S_+^{d-1}$, 
\begin{align}\label{Pf_RegPosi005}
 \int_{ \mathbb S_+^{d-1} } 
\frac{1}{ \langle f,x \rangle^{\alpha} } \pi_s(dx) 
& =  \int_{\{x: \langle f, x \rangle > e^{ -C_1 n_0} \}}
\frac{1}{ \langle f, x \rangle^{\alpha} } \pi_s(dx)  
 + \sum_{n=n_0}^{\infty} \int_{B_{f,n}} \frac{1}{ \langle f,x \rangle^{\alpha}} \pi_s(dx) \nonumber\\
& \leq    e^{ \alpha C_1 n_0 } 
+ \sum_{n=n_0}^{\infty} e^{ \alpha C_1} e^{-( C_2 - \alpha C_1 ) n } < +\infty. 
\end{align}
This concludes the proof of \eqref{RegularityIne00}. 
Using the Markov inequality, we can easily deduce \eqref{RegularityIne} from \eqref{RegularityIne00}. 
\end{proof}

\subsection{Proof of Theorem \ref{Thm-BerryE-Posi_Weak}}
It turns out that the H\"{o}lder regularity of the stationary measure $\nu$ established in subsection \ref{subsec_HolderRegu}
plays a crucial role for proving Theorem \ref{Thm-BerryE-Posi_Weak}. 

\begin{proof}[Proof of Theorem \ref{Thm-BerryE-Posi_Weak}]
Without loss of generality, we assume that the target function $\varphi$ is non-negative. 
We first prove the Berry-Esseen type bound \eqref{BerryE01_Weak} for the scalar product $\langle f, G_n x \rangle$. 

The lower bound has been shown in \eqref{BE_Low_fff}. 
The upper bound is a consequence of  \eqref{Pf_RegPosi001} 
together with the Berry-Esseen bound \eqref{Norm-BerryE-Posi}.  
In fact, using \eqref{Pf_RegPosi001} with $s =0$, we get that 
there exist constants $C_1, C_2>0$ and $k_0 \in \bb{N}$ such that for all $n \geq k \geq k_0$, 
\begin{align}\label{Regu_n_k_01}
& \bb P \left( \langle f, X_n^x \rangle \leq e^{-C_1 k} \right)  \nonumber\\
& \leq \int \bb P \left( \langle f, (g_n \ldots g_{n-k+1}) \cdot X_{n-k}^x \rangle \leq e^{-C_1 k} \right)  
\mu(dg_1) \ldots \mu(dg_{n-k})  \nonumber\\
& \leq e^{-C_2 k}. 
\end{align}
It follows that
\begin{align*}
& \mathbb{E} \Big[ \varphi(X_n^x)
\bbm{1}_{ \big\{ \frac{\log \langle f, G_n x \rangle - n \lambda }{ \sigma \sqrt{n} } \leq y \big\} }
\Big]    \nonumber\\
 & \leq    \bb{E} \Big[ \varphi(X_n^x)
 \bbm{1}_{ \big\{ \frac{\log \langle f, G_n x \rangle - n \lambda }{ \sigma \sqrt{n} } \leq y \big\} }
    \bbm{1}_{ \big\{ \log \langle f, X_n^x \rangle >  - C_1 k  \big\} }  \Big] 
    + e^{-C_2 k} \| \varphi \|_{\infty}    \nonumber\\
  &  \leq    \bb{E} \Big[ \varphi(X_n^x)
 \bbm{1}_{ \big\{ \frac{\log \|G_n \| - C_1 k - n \lambda }{\sigma \sqrt{n}} \leq y \big\} } \Big]
 + e^{-C_2 k} \| \varphi \|_{\infty}.
\end{align*}
Taking $k = \floor{C_3 \log n}$ with 
$C_3 =  \frac{1}{2C_2}$, we get that $e^{-C_2 k} \leq \frac{C}{\sqrt{n}}$ for some constant $C>0$. 
Using the Berry-Esseen bound \eqref{Norm-BerryE-Posi} with $y$ replaced by $y_1:= y+ \frac{C_1 k }{\sigma \sqrt{n}}$, 
we obtain the following upper bound: there exists a constant $C>0$ such that for all
$x \in \bb S_+^{d-1}$, $y \in \bb{R}$, $\varphi \in \mathcal{B}_{\gamma}$, and $n\geq k_0$ with $k_0$ large enough, 
\begin{align*}
\mathbb{E} \Big[ \varphi(X_n^x)
\bbm{1}_{ \big\{ \frac{\log \langle f, G_n x \rangle - n \lambda }{ \sigma \sqrt{n} } \leq y \big\} }
\Big]  
\leq  \nu(\varphi) \Phi(y_1) +  \frac{C \log n}{\sqrt{n}}  \| \varphi \|_{\gamma}. 
\end{align*}
By calculations similar to \eqref{BE_Calcu_kkk}, it can be seen that for any $y \in \bb R$, 
\begin{align*}
\Phi(y_1)  \leq  \Phi(y) +  \frac{C \log n}{\sqrt{n}}. 
\end{align*}
This concludes the proof of \eqref{BerryE01_Weak}.

Using \eqref{BerryE01_Weak} together with the Collatz-Wielandt formula, 
the proof of \eqref{BerryE02_Weak} can be carried out in 
the same way as that of \eqref{BerryE02}. We omit the details.  
\end{proof}

\section{Proofs of moderate deviation expansions}

The aim of this section is to establish
Theorems \ref{Thm-Cram-Norm}, \ref{Thm-Cram-Posi-tag} and \ref{ciro-LDP002}  
on moderate deviation asymptotics, 
and Proposition \ref{Prop_Variance} about the expressions of the asymptotic variance $\sigma^2$. 

\subsection{Proof of Theorems \ref{Thm-Cram-Norm} and \ref{Thm-Cram-Posi-tag}}
To establish Theorems \ref{Thm-Cram-Norm} and \ref{Thm-Cram-Posi-tag}, 
we need the following 
Cram\'{e}r type moderate deviation expansion for the norm cocycle $\log | G_n x|$.  

\begin{lemma} \label{MainThmNormTarget}
Assume conditions \ref{Condi-MomentH}, \ref{CondiAPH} and \ref{Condi-VarianceH}. 
Then, as $n \to \infty$, we have, 
uniformly in $x\in \mathbb{S}_+^{d-1}$, $y \in [0, o(\sqrt{n} )]$ and $\varphi \in \mathcal{B}_{\gamma}$,
\begin{align}\label{ThmNormTarget01}
  \frac{\mathbb{E}
\left[ \varphi(X_n^x) \mathbbm{1}_{ \{ \log|G_nx| - n \lambda \geq \sqrt{n} \sigma y \} } \right] }
{ 1-\Phi(y) }
= e^{ \frac{y^3}{\sqrt{n}}\zeta ( \frac{y}{\sqrt{n}} ) }
\Big[ \nu(\varphi) + \|\varphi \|_{\gamma} O \Big( \frac{y+1}{\sqrt{n}} \Big) \Big], 
\end{align}
\begin{align}\label{ThmNormTarget02}
  \frac{\mathbb{E}
\left[ \varphi(X_n^x) \mathbbm{1}_{ \{ \log|G_nx| - n \lambda \leq - \sqrt{n} \sigma y \} } \right] }
{ \Phi(-y)  }
= e^{ - \frac{y^3}{\sqrt{n}}\zeta (-\frac{y}{\sqrt{n}} ) }
\Big[ \nu(\varphi) + \|\varphi \|_{\gamma} O\Big( \frac{y+1}{\sqrt{n}} \Big) \Big].
\end{align}
\end{lemma}

Lemma \ref{MainThmNormTarget} has been recently established in \cite{XGL19b}
by developing a new smoothing inequality, applying a saddle point method 
and spectral gap properties of the transfer operator corresponding to the Markov chain $(X_n^x)_{n \geq 0}$. 
Note that condition \ref{Condi-KestenH} is not assumed in Lemma \ref{MainThmNormTarget}
and the expansions \eqref{ThmNormTarget01} and \eqref{ThmNormTarget02} 
hold uniformly with respect to the starting point $x$ 
on the whole projective space $\mathbb{S}_+^{d-1}$. 

We now prove Theorem \ref{Thm-Cram-Norm} using Theorem \ref{Thm-BerryE-Norm}, 
Lemmas \ref{Lem_Norm} and \ref{MainThmNormTarget}. 

\begin{proof}[Proof of Theorem \ref{Thm-Cram-Norm}]
Without loss of generality, we assume that $\varphi$ is non-negative on $\mathbb{S}_+^{d-1}$. 

We first prove the moderate deviation expansion \eqref{CramerNorm01}. 
In the case where $y \in [0,1]$, 
the expansion \eqref{CramerNorm01} follows from the Berry-Esseen bound \eqref{BerryE-Posi02} together with the fact that 
there exists a constant $C>0$ such that 
for all $n \geq 1$ and $\varphi \in \mathcal{B}_{\gamma}$, 
\begin{align} \label{Ch5_SP_ddd}
\sup_{ x \in \mathbb{S}_+^{d-1} } 
\Big| \mathbb{E} \big[ \varphi(X_n^x) \big]
 - \nu(\varphi) \Big| \leq  \frac{C}{\sqrt{n}}  \| \varphi \|_{\gamma}.
\end{align}
It remains to establish the expansion \eqref{CramerNorm01}
in the case where $y \in (1, o(\sqrt{n} )]$. 
The proof consists of lower and upper bounds.

The lower bound is an easy consequence of Lemma \ref{MainThmNormTarget}. 
In fact, using the expansion \eqref{ThmNormTarget01} together with the fact $\log \| G_n \| \geq \log |G_n x|$,  
there exists a constant $C>0$ such that for all $n \geq 1$, 
$x \in \mathbb{S}_+^{d-1}$, $y \in (1, o(\sqrt{n} )]$ and $\varphi \in \mathcal{B}_{\gamma}$,
\begin{align}\label{Pf_Norm_Low}
  \frac{\mathbb{E}
\left[ \varphi(X_n^x)
  \mathbbm{1}_{ \{ \log \| G_n \| - n \lambda \geq \sqrt{n} \sigma y \} } \right] }
{ 1-\Phi(y)  }
\geq e^{ \frac{y^3}{\sqrt{n}}\zeta (\frac{y}{\sqrt{n}} ) }
\Big[ \nu(\varphi) - C \frac{y+1}{\sqrt{n}} \|\varphi \|_{\gamma}  \Big].
\end{align}

The upper bound can be deduced from Lemmas \ref{Lem_Norm} and \ref{MainThmNormTarget}. 
From Lemma \ref{Lem_Norm}, we have seen that the inequality \eqref{Posi-NormLower01} holds for some constant $C_1 >0$.
For any $y \in (1, o(\sqrt{n} )]$, we denote 
\begin{align*}
y_1 = y - \frac{C_1}{ \sigma \sqrt{n} }. 
\end{align*} 
Since $y_1 \in [0, o(\sqrt{n} )]$ for sufficiently large $n$,
we are allowed to apply the moderate deviation expansion \eqref{ThmNormTarget01}
with $y$ replaced by $y_1$. 
Specifically, using \eqref{ThmNormTarget01} and \eqref{Posi-NormLower01},
we obtain that for any compact set $K \subset (\mathbb{S}_+^{d-1})^{\circ}$, there exists a constant $C>0$ such that,
as $n \to \infty$, uniformly in $x \in K$,  $y \in (1, o(\sqrt{n} )]$ and $\varphi \in \mathcal{B}_{\gamma}$,  
\begin{align}\label{Crame_Norm_Low02}
\frac{\mathbb{E}
\left[ \varphi(X_n^x) \mathbbm{1}_{ \{ \log \| G_n \| - n \lambda \geq \sqrt{n} \sigma y \} } \right] }
{ 1-\Phi(y_1) }
&  \leq  \frac{\mathbb{E}
\left[ \varphi(X_n^x) \mathbbm{1}_{ \{ \log |G_n x| - n \lambda \geq \sqrt{n} \sigma y_1 \} } \right] }
{ 1-\Phi(y_1) }      \nonumber\\
&  \leq  e^{ \frac{y_1^3}{\sqrt{n}}\zeta(\frac{y_1}{\sqrt{n}} ) }
\Big[ \nu(\varphi) + C \frac{y_1 + 1}{\sqrt{n}} \|\varphi \|_{\gamma}  \Big].
\end{align}
Since the Cram\'{e}r series $\zeta$ is convergent and analytic in a small neighborhood of $0$,
there exist constants $c, C>0$ such that for all $y \in (1, o(\sqrt{n} )]$, 
\begin{align}\label{Pf_Cram_Ine_kkk}
\left| \zeta( \frac{y_1}{\sqrt{n}} ) - \zeta(\frac{y}{\sqrt{n}} ) \right| 
\leq c \frac{|y_1 - y|}{\sqrt{n}}
\leq \frac{ C }{ n }.  
\end{align}
By simple calculations, it follows that uniformly in $y \in (1, o(\sqrt{n} )]$,
\begin{align}\label{Pf_Cram_Ine_eee}
& \exp \left\{  \frac{y_1^3}{\sqrt{n}}\zeta(\frac{y_1}{\sqrt{n}} ) 
               - \frac{y^3}{\sqrt{n}}\zeta(\frac{y}{\sqrt{n}} ) \right\} \nonumber\\
& =  \exp \left\{ \Big[ \frac{y_1^3}{\sqrt{n}} - \frac{y^3}{\sqrt{n}} \Big]  \zeta(\frac{y_1}{\sqrt{n}} ) 
                 \right\}
   \exp \left\{  \frac{y^3}{\sqrt{n}} \Big[ \zeta(\frac{y_1}{\sqrt{n}} ) - \zeta(\frac{y}{\sqrt{n}} ) \Big]
               \right\}   \nonumber\\
& =  \exp \left\{ \Big[ - \frac{3C_1}{\sigma} \frac{y^2}{n} + \frac{3 C_1^2}{\sigma^2} \frac{y}{n^{3/2}}
                       - \frac{C_1^3}{\sigma^3} \frac{1}{n^2}  \Big]  \zeta(\frac{y_1}{\sqrt{n}} ) 
                 \right\}     \nonumber\\
& \quad  \times   \exp \left\{  \frac{y^3}{\sqrt{n}} \Big[ \zeta(\frac{y_1}{\sqrt{n}} ) - \zeta(\frac{y}{\sqrt{n}} ) \Big]
               \right\}   \nonumber\\
& \leq  \exp \left\{  C_2 \Big( \frac{y^2}{n} + \frac{1}{n^2} \Big) \right\}
        \exp \left\{  C_3 \frac{y^3}{n^{3/2}}  \right\}   \nonumber\\
& \leq  1 + C_4 \frac{y^2 + 1}{ n }. 
\end{align}
Note that   
\begin{align*}
  \frac{1 - \Phi(y_1) }{1 - \Phi(y) } 
          = 1 + \left( \int_{ y -  \frac{C_1}{\sigma \sqrt{n}} }^{y} e^{- \frac{t^2}{2}} dt \right)
                   \left( \int_y^{\infty} e^{- \frac{t^2}{2}} dt \right)^{-1}.  
\end{align*}
Using the basic inequality 
$\frac{y}{y^2 + 1} e^{ - \frac{y^2}{2} } < \int_y^{\infty} e^{- \frac{t^2}{2}} dt$, $y > 1$, 
we obtain that uniformly in $y \in (1, o(\sqrt{n} )]$,
\begin{align*}
1 < \frac{1 - \Phi(y_1) }{1 - \Phi(y) }   & <  1 + \frac{y^2 + 1}{y} e^{ \frac{y^2}{2} }    
        \frac{C_1}{\sigma \sqrt{n}} e^{ -\frac{1}{2} (y -  \frac{C_1}{\sigma \sqrt{n}})^2 }   \nonumber\\
 & =  1 + (y + \frac{1}{y})  \frac{C_1}{\sigma \sqrt{n}}  
      e^{ \frac{C_1 y}{ \sigma \sqrt{n} } - \frac{C_1^2}{ 2 \sigma^2 n } }  
  = 1 + O( \frac{y}{ \sqrt{n} } ). 
\end{align*}
This implies that uniformly in $y \in (1, o(\sqrt{n} )]$,
\begin{align}\label{Crame_Norm_Low03}
\frac{1 - \Phi(y_1) }{1 - \Phi(y) } = 1+ O(\frac{y+1}{ \sqrt{n} }).
\end{align}
Note that $\frac{y_1 + 1}{\sqrt{n}} = O (\frac{y + 1}{\sqrt{n}})$.
Combining this with \eqref{Crame_Norm_Low02}, \eqref{Pf_Cram_Ine_eee} and \eqref{Crame_Norm_Low03}, 
we obtain that there exists a constant $C>0$ such that,
as $n \to \infty$, uniformly in $x \in K$,  $y \in (1, o(\sqrt{n} )]$ and $\varphi \in \mathcal{B}_{\gamma}$,  
\begin{align*}
  \frac{\mathbb{E}
\left[ \varphi(X_n^x) \mathbbm{1}_{ \{ \log \| G_n \| - n \lambda \geq \sqrt{n} \sigma y \} } \right] }
{ 1-\Phi(y)  }
\leq e^{ \frac{y^3}{\sqrt{n}}\zeta(\frac{y}{\sqrt{n}} ) }
\Big[ \nu(\varphi) + C \frac{y+1}{ \sqrt{n} } \|\varphi \|_{\gamma} \Big].
\end{align*}
Together with \eqref{Pf_Norm_Low}, 
this concludes the proof of the expansion \eqref{CramerNorm01}.

\medskip

We next prove the moderate deviation expansion \eqref{CramerNorm02}.
The proof consists of upper and lower bounds.

For the upper bound, in a similar way as in the proof of \eqref{Pf_Norm_Low}, 
using the expansion \eqref{ThmNormTarget02} together with the fact $\log \| G_n \| \geq \log |G_n x|$, 
we immediately get that there exists a constant $C>0$ such that, as $n \to \infty$, 
uniformly in $x \in \mathbb{S}_+^{d-1}$, $y \in [0, o(\sqrt{n} )]$ and $\varphi \in \mathcal{B}_{\gamma}$,
\begin{align}\label{Pf_Cra_Low_nn}
\frac{\mathbb{E}
  \left[ \varphi(X_n^x) \mathbbm{1}_{ \{ \log \| G_n \| - n \lambda \leq - \sqrt{n} \sigma y \} } \right] }
  { \Phi(-y) }
    \leq  e^{ - \frac{y^3}{\sqrt{n}}\zeta (-\frac{y}{\sqrt{n}} ) }
   \Big[ \nu(\varphi) +  \|\varphi \|_{\gamma} O \Big( \frac{ y+1 }{ \sqrt{n} } \Big) \Big].
\end{align}

For the lower bound, recall that by Lemma \ref{Lem_Norm}, 
the inequality \eqref{Posi-NormLower01} holds for some constant $C_1 >0$. 
For any $y \in [0, o(\sqrt{n} )]$, we denote 
\begin{align*}
y_2 = y + \frac{C_1}{ \sigma \sqrt{n} }, 
\end{align*} 
and it holds that $y_2 \in [0, o(\sqrt{n} )]$. 
Applying the inequality \eqref{Posi-NormLower01} 
and the moderate deviation expansion \eqref{ThmNormTarget01} with $y$ replaced by $y_2$,
we obtain that for any compact set $K \subset (\mathbb{S}_+^{d-1})^{\circ}$, there exists a constant $C>0$ such that,
as $n \to \infty$, uniformly in $x \in K$,  $y \in [0, o(\sqrt{n} )]$ and $\varphi \in \mathcal{B}_{\gamma}$,  
\begin{align}\label{Cram_Norm_dff}
\frac{\mathbb{E}
  \left[ \varphi(X_n^x) \mathbbm{1}_{ \{ \log \| G_n \| - n \lambda \leq - \sqrt{n} \sigma y \} } \right] }
  { \Phi(-y_2) }
  &  \geq  \frac{\mathbb{E}
  \left[ \varphi(X_n^x) \mathbbm{1}_{ \{ \log | G_n x| - n \lambda \leq - \sqrt{n} \sigma y_2 \} } \right] }
  { \Phi(-y_2) }   \nonumber\\
  & \geq   e^{ - \frac{y_2^3}{\sqrt{n}}\zeta (-\frac{y_2}{\sqrt{n}} ) }
   \Big[ \nu(\varphi) +  \|\varphi \|_{\gamma} O \Big( \frac{ y_2 +1 }{ \sqrt{n} } \Big) \Big].
\end{align}
Similarly to \eqref{Pf_Cram_Ine_kkk} and \eqref{Pf_Cram_Ine_eee}, by simple calculations, 
we get that uniformly in $y \in [0, o(\sqrt{n} )]$,
\begin{align}\label{Pf_Ine_2_kkk}
\left| \zeta( -\frac{y}{\sqrt{n}} ) - \zeta( -\frac{y_2}{\sqrt{n}} )  \right| 
\leq c \frac{|y_2 - y|}{\sqrt{n}}
\leq \frac{ C }{ n },    
\end{align}
and 
\begin{align}\label{Pf_Cram_2_eee}
& \exp \left\{  \frac{y^3}{ \sqrt{n} }\zeta( -\frac{y}{\sqrt{n}} ) 
               - \frac{y_2^3}{ \sqrt{n} }\zeta( -\frac{y_2}{\sqrt{n}} ) \right\} \nonumber\\
& =  \exp \left\{ \Big[ \frac{y^3}{\sqrt{n}} - \frac{y_2^3}{\sqrt{n}} \Big]  \zeta(-\frac{y}{\sqrt{n}} ) 
                 \right\}
   \exp \left\{  \frac{y_2^3}{\sqrt{n}} \Big[ \zeta( -\frac{y}{\sqrt{n}} ) - \zeta( -\frac{y_2}{\sqrt{n}} ) \Big]
               \right\}   \nonumber\\
& =  \exp \left\{ \Big[ - \frac{3C_1}{\sigma} \frac{y^2}{n} - \frac{3 C_1^2}{\sigma^2} \frac{y}{n^{3/2}}
                       - \frac{C_1^3}{\sigma^3} \frac{1}{n^2}  \Big]  \zeta( -\frac{y}{\sqrt{n}} ) 
                 \right\}   \nonumber \\
& \quad \times   \exp \left\{  \frac{y_2^3}{\sqrt{n}} \Big[ \zeta( -\frac{y}{\sqrt{n}} ) - \zeta( -\frac{y_2}{\sqrt{n}} ) \Big]
               \right\}  \nonumber\\
& \geq  \exp \left\{ - C_2 \frac{y^2 + 1}{n}  \right\}
        \exp \left\{ - C_3 \frac{y^3 + 1}{n^{3/2}}  \right\}   \nonumber\\
& \geq  1 - C_4 \frac{y^2 + 1}{ n }. 
\end{align}
Notice that   
\begin{align*}
  \frac{\Phi(-y_2) }{ \Phi(-y) } 
  =  \frac{ 1 - \Phi(y_2) }{ 1 - \Phi(y) }
     = 1 - \left( \int_{ y }^{ y + \frac{C_1}{\sigma \sqrt{n}} } e^{- \frac{t^2}{2}} dt \right)
                 \left( \int_y^{\infty} e^{- \frac{t^2}{2}} dt \right)^{-1}.  
\end{align*}
It is easy to see that $1 > \frac{\Phi(-y_2) }{ \Phi(-y) }  > 1 - \frac{C}{\sqrt{n}}$, uniformly in $y \in [0,1]$. 
From the basic inequality 
$\frac{y}{y^2 + 1} e^{ - \frac{y^2}{2} } < \int_y^{\infty} e^{- \frac{t^2}{2}} dt$, $y > 1$, 
we deduce that uniformly in $y \in (1, o(\sqrt{n} )]$,
\begin{align*}
1 > \frac{\Phi(-y_2) }{ \Phi(-y) }  &  >  1 - \frac{y^2 + 1}{y} e^{ \frac{y^2}{2} }    
        \frac{C_1}{\sigma \sqrt{n}} e^{ -\frac{1}{2} (y + \frac{C_1}{\sigma \sqrt{n}})^2 }   \nonumber\\
 & =  1 - (y + \frac{1}{y})  \frac{C_1}{\sigma \sqrt{n}}  
      e^{ -\frac{C_1 y}{ \sigma \sqrt{n} } - \frac{C_1^2}{ 2 \sigma^2 n } }  
  = 1 + O( \frac{y}{ \sqrt{n} } ). 
\end{align*}
Hence we get that uniformly in $y \in [0, o(\sqrt{n} )]$,
\begin{align}\label{Crame_Norm_hh_03}
\frac{\Phi(-y_2) }{ \Phi(-y) } = 1 + O(\frac{y+1}{ \sqrt{n} }).
\end{align}
Note that $\frac{y_2 + 1}{\sqrt{n}} = O (\frac{y + 1}{\sqrt{n}})$.
Combining this with \eqref{Cram_Norm_dff}, \eqref{Pf_Cram_2_eee} and \eqref{Crame_Norm_hh_03}, 
we obtain that there exists a constant $C>0$ such that,
as $n \to \infty$, uniformly in $x \in K$,  $y \in [0, o(\sqrt{n} )]$ and $\varphi \in \mathcal{B}_{\gamma}$,  
\begin{align*}
\frac{\mathbb{E}
  \left[ \varphi(X_n^x) \mathbbm{1}_{ \{ \log \| G_n \| - n \lambda \leq - \sqrt{n} \sigma y \} } \right] }
  { \Phi(-y_2) }
\geq  e^{ - \frac{y^3}{\sqrt{n}}\zeta (-\frac{y}{\sqrt{n}} ) }
\Big[ \nu(\varphi) + C \frac{y+1}{ \sqrt{n} } \|\varphi \|_{\gamma} \Big].
\end{align*}
This, together with the upper bound \eqref{Pf_Cra_Low_nn}, 
concludes the proof of the moderate deviation expansion \eqref{CramerNorm02}.
\end{proof}

We next prove Theorem \ref{Thm-Cram-Posi-tag}
based on Lemmas \ref{lem equiv Kesten} and \ref{MainThmNormTarget}.

\begin{proof}[Proof of Theorem \ref{Thm-Cram-Posi-tag}]

Without loss of generality, we assume that $\varphi$ is non-negative on $\mathbb{S}_+^{d-1}$.
We first prove \eqref{CramerThm01}.
The proof consists of upper and lower bounds.

\textit{Upper bound.}
Since $\log \langle f, G_n x \rangle \leq \log |G_n x|$,  
applying Lemma \ref{MainThmNormTarget}, this implies that there exists a constant $C>0$ such that,
as $n \to \infty$, 
uniformly in $f, x \in \mathbb{S}_+^{d-1}$, $y \in [0, o(\sqrt{n} )]$ and $\varphi \in \mathcal{B}_{\gamma}$,
\begin{align}\label{Pf-Scal-Upp}
  \frac{\mathbb{E}
\left[ \varphi(X_n^x)
  \mathbbm{1}_{ \{ \log \langle f, G_nx \rangle - n \lambda \geq \sqrt{n} \sigma y \} } \right] }
{ 1-\Phi(y)  }
\leq e^{ \frac{y^3}{\sqrt{n}}\zeta (\frac{y}{\sqrt{n}} ) }
\Big[ \nu(\varphi) + C \frac{y+1}{\sqrt{n}} \|\varphi \|_{\gamma} \Big].
\end{align}

\textit{Lower bound.}
Using \eqref{Posi-ScalLower01} and applying \eqref{ThmNormTarget01} in Lemma \ref{MainThmNormTarget},
with $y_1 = y + \frac{C_1}{ \sigma \sqrt{n} } $, we obtain
that there exists a constant $c>0$ such that
\begin{align}\label{Posi-ScalLower02}
\frac{\mathbb{E}
\left[ \varphi(X_n^x) \mathbbm{1}_{ \{ \log | \langle f, G_n x \rangle| - n \lambda \geq \sqrt{n} \sigma y \} } \right] }
{ 1-\Phi(y_1) }
\geq e^{ \frac{y_1^3}{\sqrt{n}}\zeta(\frac{y_1}{\sqrt{n}} ) }
\Big[ \nu(\varphi) - c \frac{y_1 + 1}{\sqrt{n}}  \|\varphi \|_{\gamma}  \Big].
\end{align}
In an analogous way as in the proof of the upper bound in Theorem \ref{Thm-Cram-Norm}, 
one can verify that $| \zeta( \frac{y_1}{\sqrt{n}} ) - \zeta(\frac{y}{\sqrt{n}} ) | \leq \frac{ C }{ n }$,
uniformly in $y \in [0, o(\sqrt{n} )]$.
Moreover, elementary calculations yield that uniformly in $y \in [0, o(\sqrt{n} )]$,
it holds that 
$e^{ \frac{y_1^3}{\sqrt{n}}\zeta(\frac{y_1}{\sqrt{n}} )
   - \frac{y^3}{\sqrt{n}}\zeta(\frac{y}{\sqrt{n}} ) }
= 1 + O(\frac{y^2 + 1}{ n })$,
$\frac{1 - \Phi(y_1) }{1 - \Phi(y) }
= 1+ O(\frac{y+1}{ \sqrt{n} })$
and 
$\frac{y_1 + 1}{\sqrt{n}} = O (\frac{y + 1}{\sqrt{n}})$.
Combining this with \eqref{Posi-ScalLower02}, we obtain
\begin{align*}
  \frac{\mathbb{E}
\left[ \varphi(X_n^x) \mathbbm{1}_{ \{ \log |\langle f, G_nx \rangle| - n \lambda \geq \sqrt{n} \sigma y \} } \right] }
{ 1-\Phi(y)  }
\geq e^{ \frac{y^3}{\sqrt{n}}\zeta(\frac{y}{\sqrt{n}} ) }
\Big[ \nu(\varphi) - c \frac{y+1}{ \sqrt{n} }  \|\varphi \|_{\gamma}  \Big].
\end{align*}
Together with the upper bound \eqref{Pf-Scal-Upp}, this concludes the proof of \eqref{CramerThm01}.
The proof of \eqref{CramerThm02} is similar to that of \eqref{CramerThm01} by using \eqref{ThmNormTarget02}
and Lemma \ref{lem equiv Kesten}.

The proof of the expansions \eqref{CramerThm03} and \eqref{CramerThm04} for the spectral radius $\rho(G_n)$ 
can be carried out in an analogous way 
using Theorem \ref{Thm-Cram-Norm}, Lemma \ref{MainThmNormTarget} and  inequality \eqref{Ine_Spectral01}. 
We omit the details. 
\end{proof}

\subsection{Proof of Theorem \ref{ciro-LDP002}}
We establish Theorem \ref{ciro-LDP002} 
on the moderate deviation principles for the entry $G_n^{i,j}$ and the spectral radius $\rho(G_n)$. 
Under conditions \ref{Condi-MomentH}, \ref{Condi-KestenH} and \ref{Condi-VarianceH}, 
the results are direct consequences of Theorem \ref{Thm-Cram-Posi-tag}. 
Under conditions \ref{Condi-MomentH}, \ref{Condi-WeakKest} and \ref{Condi-NonArith},
the proof relies on the H\"{o}lder regularity of the stationary measure $\nu$ 
shown in Proposition \ref{PropRegularity}. 

\begin{proof}[Proof of Theorem \ref{ciro-LDP002}]
As mentioned above, it remains to establish Theorem \ref{ciro-LDP002} under 
conditions \ref{Condi-MomentH}, \ref{Condi-WeakKest} and \ref{Condi-NonArith}.
We first prove the assertion (1) on the moderate deviation principle 
for the scalar product $\langle f, G_n x \rangle$. 

Let $\varphi \in \mathcal{B}_{\gamma}$ be any real-valued function satisfying $\nu(\varphi) > 0$. 
By Lemma 4.4 of \cite{HL12}, 
it suffices to prove the following moderate deviation asymptotics: 
for any $y >0$, uniformly in $f, x \in \bb S_+^{d-1}$, 
\begin{align}
&  \lim_{n\to \infty} \frac{n}{b_n^{2}}
\log   
\bb{E}  \Big[   \varphi(X_n^x) 
\bbm{1}_{ \big\{ \frac{\log \langle f, G_n x \rangle  - n\lambda }{b_n} \geq y  \big\} }   \Big] 
=  - \frac{y^2}{2\sigma^2},    \label{MDP_Entry_y_Posi}   \\
&   \lim_{n\to \infty} \frac{n}{b_n^{2}}
\log   
\bb{E}  \Big[   \varphi(X_n^x) 
\bbm{1}_{ \big\{ \frac{\log \langle f, G_n x \rangle - n\lambda }{b_n} \leq -y  \big\} }   \Big] 
=  - \frac{y^2}{2\sigma^2}.   \label{MDP_Entry_y_Neg}
\end{align}

We first prove \eqref{MDP_Entry_y_Posi}. 
The upper bound follows immediately from Lemma \ref{MainThmNormTarget} and the fact that
$\langle f, G_n x \rangle \leq | G_n x|$: 
for any $y>0$, uniformly in $f, x \in \bb S_+^{d-1}$,
\begin{align}\label{MDP_Entry_y_Posi_Upp}  
\limsup_{n\to \infty} \frac{n}{b_n^{2}}
\log   
\bb{E}  \Big[   \varphi(X_n^x) 
\bbm{1}_{  \big\{ \frac{\log \langle f, G_n x \rangle - n\lambda }{b_n} \geq y  \big\} }   \Big] 
\leq - \frac{y^2}{2\sigma^2}.  
\end{align}
The lower bound can be deduced from Lemma \ref{MainThmNormTarget} 
together with Proposition \ref{PropRegularity}.
Specifically, using \eqref{Regu_n_k_01}, 
we obtain that there exist constants $C_1, C_2 >0$ and $k_0 \in \bb{N}$ such that for all $n \geq k \geq k_0$,
\begin{align}\label{Entry_Inequa_01}
 I_n: & =  \bb{E}  \Big[   \varphi(X_n^x) 
\bbm{1}_{  \big\{ \frac{\log \langle f, G_n x \rangle - n\lambda }{b_n} \geq y  \big\} }  \Big]   \nonumber\\
 & \geq 
 \bb{E}  \Big[   \varphi(X_n^x) 
 \bbm{1}_{  \big\{ \frac{\log \langle f, G_n x \rangle - n\lambda }{b_n} \geq y  \big\} } 
  \bbm{1}_{ \big\{ \log \langle f, G_n x \rangle - \log |G_n x|  \geq  - C_1 k  \big\} }  \Big]   \nonumber\\
 & \geq  
  \bb{E}  \Big[   \varphi(X_n^x) 
 \bbm{1}_{ \big\{ \log |G_n x| - n\lambda  \geq y b_n + C_1  k  \big\} } 
  \bbm{1}_{ \big\{ \log \langle f, G_n x \rangle - \log |G_n x| \geq  - C_1 k  \big\} }  \Big]  \nonumber\\   
 & \geq   \bb{E}  \Big[   \varphi(X_n^x) 
 \bbm{1}_{ \big\{ \log |G_n x| - n\lambda  \geq y b_n + C_1 k  \big\} }  \Big] 
  - e^{-C_2 k} \|\varphi \|_{\infty}.  
\end{align}
In the sequel, we  take 
\begin{align}\label{Pf_Range_k}
k = \floor[\Big]{ C_3 \frac{b_n^2}{n} },  
\end{align}
where $C_3 >0$ is a  constant whose value will be chosen large enough. 
From the moderate deviation expansion \eqref{ThmNormTarget01}, 
it follows that for any $y>0$ and $\eta > 0$, there exists $n_0 \in \bb{N}$ such that for all $n \geq n_0$,
\begin{align}\label{Ch6_Bound_Norm_aa}
\bb{E}  \Big[   \varphi(X_n^x) 
 \bbm{1}_{ \big\{ \frac{\log | G_n x| - n\lambda }{b_n} \geq y  \big\} }   \Big] 
 \geq  e^{ - \frac{b_n^2}{n} \big(  \frac{y^2}{2 \sigma^2} + \eta \big) }. 
\end{align}
Set $$b_n' = b_n + \frac{C_1 k}{y}.$$   
We easily see that the sequence $(b_n')_{n \geq 1}$ satisfies 
$\frac{b_n'}{\sqrt{n}}\to \infty$ and $\frac{b_n'}{n} \to 0$, as $n \to \infty$. 
Using \eqref{Ch6_Bound_Norm_aa}, we get that 
uniformly in $f, x \in \bb S_+^{d-1}$,
\begin{align*}
 \bb{E}  \Big[ \varphi(X_n^x) 
 \bbm{1}_{ \{ \log |G_n x| - n\lambda  \geq y b_n + \ee k \} }  \Big]  
 \geq  e^{ - \frac{ (b_n')^2 }{n} \big(  \frac{y^2}{2 \sigma^2} + \eta \big) }. 
\end{align*} 
Substituting this into \eqref{Entry_Inequa_01},  we obtain
\begin{align*}
I_n \geq  e^{ - \frac{ (b_n')^2 }{n} \big(  \frac{y^2}{2 \sigma^2} + \eta \big) }
   \Big[ 1 -  e^{ -C_2 k + \frac{ (b_n')^2 }{n} \big( \frac{y^2}{2 \sigma^2} + \eta \big) } \|\varphi\|_{\infty}  \Big]. 
\end{align*}
In view of \eqref{Pf_Range_k}, 
choosing $C_3 >  \frac{ 1 }{C_2} \big( \frac{y^2}{2 \sigma^2} +  \eta \big)$, 
by elementary calculations, we get
\begin{align*}
\lim_{n \to \infty} \frac{ (b_n')^2 }{k n}  \big( \frac{y^2}{2 \sigma^2} +  \eta \big) 
= \frac{ 1 }{C_3}  \, \big( \frac{y^2}{2 \sigma^2} +  \eta \big) < C_2.
\end{align*}
Thus, for some constant $C_4 >0$, 
\begin{align*}
I_n \geq  e^{ - \frac{ (b_n')^2 }{n} \big(  \frac{y^2}{2 \sigma^2} + \eta \big) }
   \Big[ 1 -  e^{ -C_4 k } \|\varphi\|_{\infty}  \Big]. 
\end{align*}
Hence, recalling that $k = \floor[]{ C_3 \frac{b_n^2}{n} } \to \infty$ as $n \to \infty$, we obtain
\begin{align*}
\liminf_{ n \to \infty } \frac{n}{ b_n^2 }  \log I_n
& \geq   \lim_{ n \to \infty }  \frac{n}{ b_n^2 }  
   \Big[  - \frac{ (b_n')^2 }{n} \Big(  \frac{y^2}{2 \sigma^2} + \eta \Big) \Big]
  + \lim_{ n \to \infty }  \frac{n}{ b_n^2 } \log ( 1 - e^{- C_4 k} \|\varphi\|_{\infty}  )   \nonumber\\
&  =  \lim_{ n \to \infty }  
  \Big[ - \Big(  1 + \frac{C_1 k}{ y b_n } \Big)^2  \Big(  \frac{y^2}{ 2 \sigma^2 } + \eta  \Big)   \Big]
  + 0  \nonumber\\
&  =  - \Big(  \frac{y^2}{ 2 \sigma^2 } + \eta \Big).
\end{align*}
Letting $\eta \to 0$, the desired lower bound follows:
for any $y > 0$, uniformly in $f, x \in \bb S_+^{d-1}$, 
\begin{align*}
\liminf_{n\to \infty}\frac{n}{b_n^{2}}
\log  \bb{E} \Big[   \varphi(X_n^x) 
\bbm{1}_{ \big\{ \frac{\log \langle f, G_n x \rangle - n\lambda }{b_n} \geq y  \big\}  }   \Big] 
\geq - \frac{y^2}{2\sigma^2}.  
\end{align*}
This, together with the upper bound \eqref{MDP_Entry_y_Posi_Upp}, concludes the proof of \eqref{MDP_Entry_y_Posi}. 

\medskip

We next prove \eqref{MDP_Entry_y_Neg}. 
By \eqref{ThmNormTarget02}
and the fact that $\langle f, G_n x \rangle \leq |G_n x|$, the lower bound easily follows:
for any $y > 0$, uniformly in $f, x \in \bb S_+^{d-1}$, 
\begin{align}\label{MDP_Entry_y_Neg_Low}
\liminf_{n\to \infty}\frac{n}{b_n^{2}}
\log  \bb{E} \Big[   \varphi(X_n^x) 
\bbm{1}_{  \big\{ \frac{\log \langle f, G_n x \rangle - n\lambda }{b_n} \leq -y  \big\}  }   \Big] 
\geq - \frac{y^2}{2\sigma^2}.  
\end{align}
For the upper bound, by \eqref{Regu_n_k_01}, 
there exist constants $C_5, C_6 >0$ and $k_0 \in \bb{N}$ such that for all $n \geq k \geq k_0$,
\begin{align*}
J_n: &  = \bb{E}  \Big[   \varphi(X_n^x) 
  \bbm{1}_{ \big\{ \frac{\log \langle f, G_n x \rangle - n\lambda }{b_n} \leq -y  \big\} }   \Big]   \nonumber\\
&  =  \bb{E}  \Big[   \varphi(X_n^x) 
  \bbm{1}_{  \big\{ \frac{\log \langle f, G_n x \rangle - n\lambda }{b_n} \leq -y   \big\} } 
   \bbm{1}_{  \big\{ \log \langle f, G_n x \rangle - \log |G_n x|  \geq  - C_5 k   \big\} }  \Big]  \nonumber\\
&  \quad  +  \bb{E}  \Big[   \varphi(X_n^x) 
  \bbm{1}_{ \big\{ \frac{\log \langle f, G_n x \rangle - n\lambda }{b_n} \leq -y  \big\} } 
   \bbm{1}_{ \big\{ \log \langle f, G_n x \rangle - \log |G_n x|  <  - C_5 k   \big\} }  \Big]    \nonumber\\
& \leq  \bb{E}  \Big[   \varphi(X_n^x) 
   \bbm{1}_{  \big\{  \log |G_n x| - n \lambda  \leq  -y b_n + C_5 k  \big\} }  \Big]
   + e^{-C_6 k} \|\varphi\|_{\infty}. 
\end{align*}
As in the proof of \eqref{MDP_Entry_y_Posi}, we choose
\begin{align}\label{MDP_Norm_k_aaa}
k = \floor[\Big]{ C_7 \frac{b_n^2}{n} },  
\end{align}
where $C_7 >0$ is a constant whose value will be chosen large enough. 
From \eqref{ThmNormTarget02}, 
it follows that for any $\eta >0$, there exists $n_0 \in \bb{N}$ such that for any $n \geq n_0$,
\begin{align}\label{MDP_Norm_aaa01} 
\bb{E}  \Big[   \varphi(X_n^x) 
\bbm{1}_{ \big\{ \frac{\log | G_n x| - n\lambda }{b_n} \leq -y  \big\} }   \Big] 
\leq e^{ - \frac{b_n^2}{n} \big(  \frac{y^2}{2 \sigma^2} - \eta \big) }. 
\end{align}
Denote $b_n' = b_n - \frac{C_5 k}{y}$. Then, by \eqref{MDP_Norm_k_aaa}, it holds that 
$\frac{b_n'}{\sqrt{n}}\to \infty$ and $\frac{b_n'}{n} \to 0$, as $n \to \infty$. 
From \eqref{MDP_Norm_aaa01}, it follows that uniformly in $f, x \in \bb S_+^{d-1}$, 
\begin{align}\label{Pf_MD_Spe_ll}
J_n  \leq  e^{ - \frac{ (b_n')^2 }{ n } \big(  \frac{y^2}{2 \sigma^2} - \eta \big) }
   + e^{-C_6 k} \|\varphi\|_{\infty}.    
\end{align} 
Note that $\frac{b_n'}{b_n} \to 1$ as $n \to \infty$.   
Choosing $C_7 >  \frac{ 1 }{C_6} \big( \frac{y^2}{2 \sigma^2} -  \eta \big)$, we get
\begin{align*}
\lim_{n \to \infty} \frac{ (b_n')^2 }{k n}  \big( \frac{y^2}{2 \sigma^2} -  \eta \big) 
= \frac{ 1 }{C_7}  \, \big( \frac{y^2}{2 \sigma^2} -  \eta \big) < C_6.
\end{align*}
Hence, 
\begin{align*}  
\limsup_{n\to \infty}  \frac{n}{b_n^{2}} \log J_n 
& \leq  \limsup_{n\to \infty}  \frac{n}{b_n^{2}} 
  \log  e^{ - \frac{ (b_n')^2 }{ n } \big(  \frac{y^2}{2 \sigma^2} - \eta \big) }  \nonumber\\
& =  - \lim_{n\to \infty}  
  \Big( \frac{b_n'}{b_n} \Big)^2  \Big( \frac{y^2}{2 \sigma^2} - \eta \Big)   
=   - \Big(  \frac{y^2}{ 2 \sigma^2 } - \eta \Big).  
\end{align*} 
Since $\eta > 0$ can be arbitrary small, 
we obtain the desired upper bound:
$$\limsup_{n\to \infty}  \frac{n}{b_n^{2}} \log J_n \leq - \frac{y^2}{ 2 \sigma^2 }.$$  
Combining this with the lower bound \eqref{MDP_Entry_y_Neg_Low}, we finish the proof of \eqref{MDP_Entry_y_Neg}.

Combining  \eqref{MDP_Entry_y_Posi} and \eqref{MDP_Entry_y_Neg}, we get
the assertion (1). 
Using the assertion (1) and the Collatz-Wielandt formula, one can obtain the assertion (2). 
\end{proof}

\subsection{Proof of Proposition \ref{Prop_Variance}}

We prove Proposition \ref{Prop_Variance} based on Lemmas \ref{Lem_Norm}, \ref{lem equiv Kesten}
and the Collatz-Wielandt formula \eqref{Formu_ColWie}. 

\begin{proof}[Proof of Proposition \ref{Prop_Variance}]
We first prove part (1). For fixed $x \in K \subset (\mathbb{S}_+^{d-1})^{\circ}$, we denote 
\begin{align*}
A_n = \mathbb{E}  \Big[ \big( \log | G_n x | - n \lambda \big)^2 \Big],  \quad
B_n =  \mathbb{E} \Big[ \big( \log \| G_n \| - n \lambda \big)^2 \Big]. 
\end{align*}
Since $\frac{1}{n} A_n \to \sigma^2$ as $n \to \infty$ (see \eqref{Formu_sig}), 
it suffices to show that $\frac{1}{n} (B_n - A_n) \to 0$ as $n \to \infty$. 
Using Minkowski's inequality, we see that there exists a constant $C>0$ independent of $x \in K$ such that
\begin{align*}
\big| \sqrt{B_n} - \sqrt{A_n}  \big|  
\leq  \sqrt{ \mathbb{E} \left[ \Big( \log \frac{\| G_n \|}{ | G_n x | } \Big)^2 \right] }   
  \leq C,  
\end{align*}
where the last inequality holds by Lemma \ref{Lem_Norm}. 
Consequently, it follows that 
\begin{align}\label{Var_Ine001}
| B_n - A_n | \leq | \sqrt{B_n} - \sqrt{A_n} |  \left( | \sqrt{B_n} - \sqrt{A_n} | + 2 \sqrt{A_n} \right)  
\leq C ( C + O(\sqrt{n}) ),
\end{align}
which leads to the desired assertion in part (1). 

Now we proceed to prove part (2). Denote 
\begin{align*}
D_n = \mathbb{E} \left[ (\log \langle f, G_n x \rangle - n \lambda )^{2} \right],  \quad
E_n =  \mathbb{E} \left[ (\log \rho(G_n) - n \lambda )^{2} \right]. 
\end{align*}
As in the proof of part (1), by Minkowski's inequality, 
we have, uniformly in $f, x \in \mathbb{S}_+^{d-1}$, 
\begin{align*}
\left| \sqrt{D_n} - \sqrt{A_n}  \right| 
\leq  \sqrt{ \mathbb{E}  \Big[ \big( \log \langle f, X_n^x  \rangle \big)^2 \Big] } 
  \leq C,  
\end{align*}
where  the last inequality holds by Lemma \ref{lem equiv Kesten}.
In the same way as in the proof of \eqref{Var_Ine001}, 
one can verify that $\frac{1}{n} (D_n - A_n) \to 0$, as $n \to \infty$, uniformly in $f, x \in \mathbb{S}_+^{d-1}$.
This ends the proof of the first equality in part (2). 
To prove the second one in part (2), using again the Minkowski inequality, we have 
\begin{align*}
\left| \sqrt{E_n} - \sqrt{B_n}  \right| 
\leq  \sqrt{ \mathbb{E} \Big[ \Big( \log \frac{ \| G_n \| }{ \rho(G_n) } \Big)^2 \Big] }.   
\end{align*}
Taking into account the Collatz-Wielandt formula \eqref{Formu_ColWie} 
with $i=1$ and $x_0 = (1,1,\ldots, 1)^{\mathrm{T}}$, 
we get that $\rho(G_n) \geq \langle e_1, G_n x_0 \rangle$.  
Since $\rho(G_n) \leq \| G_n \|$ and $\| G_n \| \leq C | G_n x_0|$ (see Lemma \ref{Lem_Norm}),
it follows from Lemma \ref{lem equiv Kesten} that 
\begin{align*}
\left| \sqrt{E_n} - \sqrt{B_n}  \right| 
\leq  C +  \sqrt{ \mathbb{E} \Big[ \big( \log \langle e_1, X_n^{x_0}  \rangle \big)^2 \Big] }
  \leq C.   
\end{align*}
Together with part (1), this proves the second equality in part (2).
\end{proof}


\end{document}